\documentclass[11pt, one side,emlines]{amsart}
\usepackage{amssymb,latexsym,xy,eucal,mathrsfs, graphicx, tikz, amsmath, caption}
\textwidth=15cm \textheight=24cm \theoremstyle{plain}
\newtheorem{lem}{Lemma}[section]
\newtheorem{thm}[lem]{Theorem}
\newtheorem{theorem}[lem]{Main Theorem}
\newtheorem{proposition}[lem]{Proposition}
\newtheorem{cor}[lem]{Corollary}

\theoremstyle{definition}

\newtheorem{Def}[lem]{Definition}

\numberwithin{equation}{section} 
\begin{document}

\title[Decomposition of Augmented Cubes]{Decomposition of Augmented Cubes into Regular Connected Pancyclic Subgraphs }
\author{S. A. Kandekar, Y. M. Borse and B. N. Waphare }
\address{\rm Department of Mathematics, Savitribai Phule Pune University, Pune 411007, M.S., INDIA.}
\email{smitakandekar54@gmail.com; ymborse11@gmail.com; waphare@yahoo.com}

\maketitle
\baselineskip 20 truept
\begin{abstract}
In this paper, we consider the problem of decomposing the augmented cube $AQ_n$ into two spanning, regular, connected and pancyclic subgraphs. We prove that for  $ n \geq 4$ and  $ 2n  - 1 = n_1 + n_2 $ with $ n_1, n_2 \geq 2,$ the augmented cube $ AQ_n$ can be decomposed into two spanning subgraphs $ H_1$ and $ H_2$ such that each $ H_i$ is $n_i$-regular and $n_i$-connected. Moreover, $H_i$ is $4$-pancyclic if $ n_i \geq 3.$ 
\end{abstract}

	\section{\textbf{Introduction}}
	
	Interconnection networks play an important role in communication systems and parallel computing. Such a network is usually represented by a graph where vertices stand for its processors and edges for links between the processors. Network topology is a crucial factor for interconnection networks as it determines the performance of the networks. Many interconnection network topologies have been proposed in literature such as mesh, torus, hypercube and hypercube like structures. The $n$-dimensional hypercube $Q_n$ is a popular interconnection network topology. It is an $n$-regular, $n$-connected, vertex-transitive graph with $2^n$ vertices and has diameter $n.$ 
	
	In 2002, Chaudam and Sunitha \cite{cs} introduced a variant of the hypercube $Q_n$ called augmented cube $AQ_n.$  The graph $AQ_n$ is $(2n - 1)$-regular, $(2n - 1)$-connected, pancyclic and vertex-transitive on $2^n$ vertices.  However, the diameter of $AQ_n$ is  $\lceil n/2 \rceil$ which is almost half the diameter of $Q_n.$ Hence there is less delay in data transmission in the augmented cube network than the hypercube network. Many results have been obtained in the literature to prove that the augmented cube is a good candidate for computer network topology design; see \cite{cs, ma, sx, ss, ww}. 
	
One of the central issue in evaluating a network is to study the embedding problem. It is said that a graph $H$ can be embedded into a graph $G$ if it is isomorphic to a subgraph of $G$ and if so, while modeling a network with graph, we can apply existing algorithms for graph $H$ to the graph $G$. Cycle networks are suitable for designing simple algorithms with low communication cost. Since some parallel applications, such as those in image and signal processing, are originally designed on a cycle architecture, it is important to have effective cycle embedding in a network. A graph $G$ is {\it $r$-pancyclic} if it contains cycles of every length from $r$ to $|V(G)|.$ A 3-pancyclic graph is  {\it pancyclic.} A graph $G$ on even number of vertices is {\it bipancyclic} if $G$ is a cycle or it contains cycles of every even length from 4 to $|V(G)|.$  A lot of research has been done regarding  pancyclicity of augmented cubes; see \cite{ch, fu, hs, ml, sx, wm, ww}. 

A {\it decomposition} of a graph $G$ is a list of its subgraphs  $ H_1$, $H_2$, \dots, $H_k$ such that every edge of $G$ belongs to  $H_i$ for exactly one $i.$   
The decompositions of  hypercubes into  Hamiltonian cycles, into smaller cycles, into paths and into trees are studied in \cite{al, ba, ss, mo, wa}. The existence of a spanning, $k$-regular, $k$-connected and bipancyclic subgraph of $Q_n,$ for every $k$ with $3 \leq k \leq n,$ is proved in  \cite{sm}.  Bass and Sudborough \cite{ba} pointed out that  the decomposition of $Q_n$ into regular, spanning, isomorphic subgraphs has potential applications in construction of adaptive routing algorithms and in the area of fault tolerant computing. They  obtained a decomposition of $Q_{n},$ for even $n,$ into two spanning, $(n/2)$-regular, $2$-connected, isomorphic subgraphs of diameter $n + 2.$ Borse and Kandekar \cite{bk} proved the existence of the decomposition of $Q_n$ into two regular subgraphs whose degrees are based on the given 2-partition of $n$ and further, they are rich with respect to connectivity and cycle embedding. 

 \begin{thm} [\cite{bk}] \label{1}
 	For  $ n = n_1 + n_2 $ with $ n_1, n_2 \geq 2,$  the hypercube $Q_n$ can be decomposed into two spanning subgraphs $ H_1$ and $ H_2$ such that $ H_i$ is $n_i$-regular, $n_i$-connected and bipancyclic. 
 \end{thm}
The above result has been generalized for the decomposition of $Q_n$ based on partitions of $n$ into $k$ parts. It has been also extended to the class of the Cartesian product of even cycles; see \cite{bs, bss}.

 
 In this paper, we extend Theorem \ref{1} to  the class of augmented cubes. The following is the main theorem of the paper.	
	\begin{theorem}\label{2}
		Let  $ n \geq 4$ and  $ 2n  - 1 = n_1 + n_2 $ with $ n_1, n_2 \geq 2.$ Then the augmented cube $ AQ_n$ can be decomposed into two spanning  subgraphs $ H_1$ and $ H_2$ such that $ H_i$ is $n_i$-regular and $n_i$-connected. Moreover, $H_i$ is $4$-pancyclic if $ n_i \geq 3.$
	\end{theorem}
	In Section 2, we  provide necessary definitions and preliminary results. The special case $ n_1 = 2$ of the main theorem is proved in third section. In section 4 and 5, we consider the cases $ n_1 = 3$ and $n_1 = 4,$ respectively. We complete the proof of the main theorem using these special cases in the last section. 
 \section{ \textbf{Preliminaries}}
 The $n$-dimensional augmented cube is denoted by $AQ_n, n\geq 1.$ It can be defined recursively as follows.\\
 $AQ_1$ is a complete graph $K_2$ with vertex set $\{0, 1\}.$ For $ n \geq 2,~ AQ_n$ is obtained from two copies of the augmented cube $AQ_{n-1},$ denoted by $AQ^0_{n-1}$ and $AQ^1_{n-1},$ and adding $2^n$ edges between them as follows.\\
 Let $V(AQ^0_{n-1}) = \{ 0x_1x_2...x_{n-1} \colon x_i = 0~~ \rm{or}~~ 1\}$ and $V(AQ^1_{n-1}) = \{1y_1y_2...y_{n-1} \colon y_i = 0~~ \rm{or}~~ 1\}.$
 A vertex $x = 0x_1x_2...x_{n-1}$ of $AQ^0_{n-1}$ is joined to a vertex $y = 1y_1y_2...y_{n-1}$ of $AQ^1_{n-1}$ if and only if either 
 \begin{enumerate}
 	\item $x_i = y_i$ for $1 \leq i \leq n-1,$ in this case the edge is called hypercube edge and we set $y = x^h$ or
 	
 	\item $x_i = \overline{y_i}$ for $1 \leq i \leq n-1,$ in this case the edge is called complementary edge and we set $y = x^c.$
 \end{enumerate}

  Let $E_n^h$ and $E_n^c$ be the set of hypercube edges and complementary edges, respectively used to construct $AQ_n$ from $AQ_{n-1}.$ Then $E_n^h$ and $E_n^c$ are the perfect matchings of $AQ_n$ and further,  $AQ_n = AQ^0_{n-1} \cup AQ^1_{n-1} \cup E_n^h \cup E_n^c.$ The Augmented cubes of dimensions 1, 2 and 3 are shown in Fig. $1.$

\begin{tikzpicture}

\draw [fill=black] (0,0) circle  (.08)  node [left]  at (0,0) {$0$};
\draw [fill=black] (0,2) circle  (.08)  node [left]  at (0,2) {$1$};
\draw [style= thin] (0,0)--(0,2);


\draw [fill=black] (3,0) circle  (.08)  node [left]  at (3,0) {$00$};
\draw [fill=black] (3,2) circle  (.08)  node [left]  at (3,2) {$01$};
\draw [fill=black] (5,0) circle  (.08)  node [right]  at (5,0) {$10$};
\draw [fill=black] (5,2) circle  (.08)  node [right]  at (5,2) {$11$};
\draw [style= thin] (3,0)--(3,2);
\draw [style= thin] (3,0)--(5,0);
\draw [style= thin] (5,0)--(5,2);
\draw [style= thin] (5,0)--(3,2);
\draw [style= thin] (3,0)--(5,2);
\draw [style= thin] (5,2)--(3,2);


\draw [fill=black] (8,0) circle  (.08)  node [left]  at (8,0) {$000$};
\draw [fill=black] (8,2) circle  (.08)  node [left]  at (8,2) {$001$};
\draw [fill=black] (10,0) circle  (.08)  node [right]  at (10.1,0.05) {$010$};
\draw [fill=black] (10,2) circle  (.08)  node [right]  at (10.1,1.95) {$011$};
\draw [style= thin] (8,0)--(8,2);
\draw [style= thin] (8,0)--(10,0);
\draw [style= thin] (10,0)--(10,2);
\draw [style= thin] (10,0)--(8,2);
\draw [style= thin] (8,0)--(10,2);
\draw [style= thin] (10,2)--(8,2);

\draw [fill=black] (12,0) circle  (.08)  node [left]  at (11.9,0.05) {$100$};
\draw [fill=black] (12,2) circle  (.08)  node [left]  at (11.9,1.95) {$101$};
\draw [fill=black] (14,0) circle  (.08)  node [right]  at (14,0) {$110$};
\draw [fill=black] (14,2) circle  (.08)  node [right]  at (14,2) {$111$};
\draw [style= thin] (12,0)--(12,2);
\draw [style= thin] (12,0)--(14,0);
\draw [style= thin] (14,0)--(14,2);
\draw [style= thin] (14,0)--(12,2);
\draw [style= thin] (12,0)--(14,2);
\draw [style= thin] (14,2)--(12,2);

\draw [style= thin] (8,0)--(14,2);
\draw [style= thin] (14,0)--(8,2);
\draw [style= thin] (10,2)--(12,0);
\draw [style= thin] (10,0)--(12,2);

\draw [style= thin] (8,0)..controls(10, -0.5)..(12,0);
\draw [style= thin] (10,0)..controls(12,-0.5 )..(14,0);
\draw [style= thin] (8,2)..controls(10,2.5 )..(12,2);
\draw [style= thin] (10,2)..controls(12,2.5 )..(14,2);
\end{tikzpicture}

\hspace{0.05in} {$AQ_1$}  \hspace{1.1in } {$AQ_2$} \hspace{2.2in } {$AQ_3$}

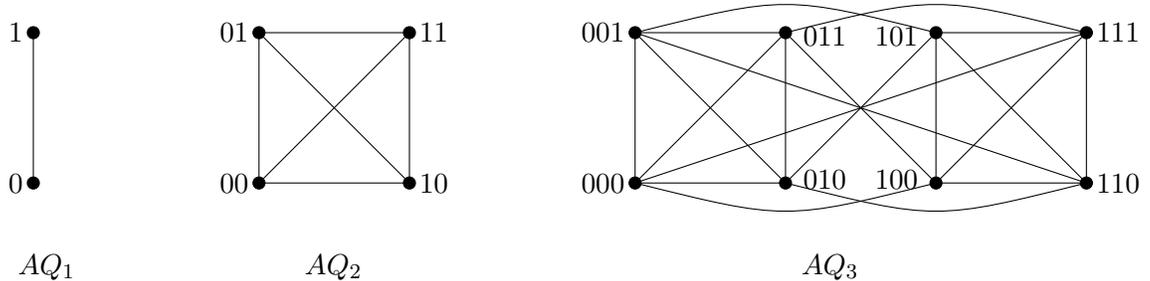
\hspace{1.2in }  {Figure 1: Augmented cube of dimension 1, 2, 3.}\\

   In $AQ_2, \, E_2^h = \{ <01, 11>, <00, 10>\}$ and $E_2^c = \{ <01, 10>, <00, 11>\}$. From the definition, it is clear that $AQ_n$ is a $(2n-1)$-regular graph on $2^n$ vertices. It is also known that $AQ_n$ is $(2n-1)$-connected and vertex-transitive \cite{cs}.

A {\it ladder} on $2m$ with $ m\geq 2$ is a graph consisting of two vertex-disjoint paths, say  $P_1 = <u_1, u_2, \dots, u_m>$ and $ P_2 = <v_1, v_2, \dots, v_m>,$ that  are joined by edges $u_iv_i$  for $ i = 1, 2, \dots m; $ see Figure 2(a). The following lemma follows easily.

\begin{lem}[\cite{bs}] \label{l1} 
	A ladder is bipancyclic.
\end{lem}

  We now define a {\it ladder-like graph} which is shown to be pancyclic and  is used to prove pancyclicity of certain  subgraphs of $AQ_n$ in subsequent results.   


%

\begin{Def}\label{d1}	
	Let $ m\geq 6$ be an integer. A subgraph $G$ of $AQ_n$ on $2m$ vertices said to be {\it a ladder-like graph} if it contains  \begin{enumerate}
		\item[(i)]two cycles, say $  Z_1= <u_1, u_2, \dots, u_m, u_1>$ and $ Z_2 = <v_1, v_2, \dots, v_m, v_1>.$
		\item[(ii)] edges $u_iv_i$  for $ i = 1, 2, \dots m$ and 
		\item[(iii)] two edges $ u_1v_4$ and $ u_4 v_1$ (see Figure 2(b)).
	\end{enumerate}  
	We say that the vertices $u_1, u_2, \dots, u_m$ are on one side of $G$ and $v_1, v_2, \dots, v_m$ are on the other side of $G.$ Denote by $B_8,$ the subgraph of $G$ induced by the set of vertices $\{ u_1, u_2, u_3, u_4, v_1, v_2, v_3, v_4\}.$ 
	
	In addition, if $G$ contains a 4-cycle $C = <u_t, u_{t+1}, v_{t+1}, v_t, u_t>, $ where $ t \geq 5$ and $ u_{t+1} = u_t^c,\,\, v_{t+1} = v_t^c,$  then the subgraph $G$ is said to be {\it a ladder-like graph with a special 4-cycle.}
\end{Def}

\begin{center}
	\begin{tikzpicture}[scale=0.9]
	\draw [fill=black] (2,0) circle  (.08)  node [left]  at (2,0) {$u_m$};
	\draw [fill=black] (2,1) circle  (.08)  node [left]  at (2,1) {$u_{5}$};
	\draw [fill=black] (2,2) circle  (.08)  node [left]  at (2,2) {$u_{4}$};
	\draw [fill=black] (2,3) circle  (.08)  node [left]  at (2,3) {$u_3$};
	\draw [fill=black] (2,4) circle  (.08)  node [left]  at (2,4) {$u_2$};
	\draw [fill=black] (2,5) circle  (.08)  node [left]  at (2,5) {$u_1$};
	
	\draw [fill=black] (5,0) circle  (.08)  node [right]  at (5,0) {$v_{m}$};
	\draw [fill=black] (5,1) circle  (.08)  node [right]  at (5,1) {$v_{5}$};
	\draw [fill=black] (5,2) circle  (.08)  node [right]  at (5,2) {$v_{4}$};
	\draw [fill=black] (5,3) circle  (.08)  node [right]  at (5,3) {$v_3$};
	\draw [fill=black] (5,4) circle  (.08)  node [right]  at (5,4) {$v_2$};
	\draw [fill=black] (5,5) circle  (.08)  node [right]  at (5,5) {$v_1$};


	\draw [style= thick, dotted] (2,0)--(2,1);
	\draw [style= thick, dotted] (5,0)--(5,1);
	\draw [style= thin] (2,1)--(2,2);
	\draw [style= thin] (5,1)--(5,2);
	\draw [style= thin] (2,2)--(2,5);
	\draw [style= thin] (5,2)--(5,5);
	\draw [style= thin] (2,0)--(5,0);
	\draw [style= thin] (2,1)--(5,1);
	\draw [style= thin] (2,2)--(5,2);
	\draw [style= thin] (2,3)--(5,3);
	\draw [style= thin] (2,4)--(5,4);
	\draw [style= thin] (2,5)--(5,5);
	
	
	\draw [fill=black] (8,0) circle  (.08)  node [left]  at (8,0) {$u_m$};
	\draw [fill=black] (8,1) circle  (.08)  node [left]  at (8,1) {$u_{5}$};
	\draw [fill=black] (8,2) circle  (.08)  node [left]  at (8,2) {$u_4$};
	\draw [fill=black] (8,3) circle  (.08)  node [left]  at (8,3) {$u_3$};
	\draw [fill=black] (8,4) circle  (.08)  node [left]  at (8,4) {$u_2$};
	\draw [fill=black] (8,5) circle  (.08)  node [left]  at (8,5) {$u_1$};
	\draw  node [above]  at (6.9,2.2) {$Z_1$};
	
	\draw [fill=black] (11,0) circle  (.08)  node [right]  at (11,0) {$v_m$};
	\draw [fill=black] (11,1) circle  (.08)  node [right]  at (11,1) {$v_{5}$};
	\draw [fill=black] (11,2) circle  (.08)  node [right]  at (11,2) {$v_4$};
	\draw [fill=black] (11,3) circle  (.08)  node [right]  at (11,3) {$v_3$};
	\draw [fill=black] (11,4) circle  (.08)  node [right]  at (11,4) {$v_2$};
	\draw [fill=black] (11,5) circle  (.08)  node [right]  at (11,5) {$v_1$};
	\draw  node [above]  at (12.1,2.2) {$Z_2$};
	
	\draw [style= thick, dotted] (8,0)--(8,1);
	\draw [style= thick, dotted] (11,0)--(11,1);
	\draw [style= thin] (8,1)--(8,2);
	\draw [style= thin] (11,1)--(11,2);
	\draw [style= thin] (8,2)--(8,5);
	\draw [style= thin] (11,2)--(11,5);
	\draw [style= thin] (8,0)--(11,0);
	\draw [style= thin] (8,1)--(11,1);
	\draw [style= thin] (8,2)--(11,2);
	\draw [style= thin] (8,3)--(11,3);
	\draw [style= thin] (8,4)--(11,4);
	\draw [style= thin] (8,5)--(11,5);
	
	\draw [style= thin] (8,0)..controls(7, 2)..(8,5);
	\draw [style= thin] (11,0)..controls(12, 2)..(11,5);
	
	\draw [style= thin] (8,5)--(11,2);
	\draw [style= thin] (8,2)--(11,5);
	\end{tikzpicture}
	\end{center}

\hspace{1.2in} {(a). Ladder} \hspace{.9in} {(b). Ladder-like graph}

\hspace{1.2in} {Figure 2: Ladder and Ladder-like graph}

A spanning ladder-like subgraph $L$ of the augmented cube $AQ_4$ is shown in Figure 4(b) by dark lines. $Z_1 = <0010,\,\, 1101,\,\, 1110,\,\, 0110,\,\, 0100,\,\, 1011,\,\, 1000,\,\, 0000,\,\, 0010>$ and\\$Z_2 = <0101,\,\, 1010,\,\, 1001,\,\, 0001,\,\, 0011,\,\, 1100,\,\, 1111,\,\, 0111,\,\, 0101>$ are the cycles of $L.$ The subgraph $B_8$ of $L$ is induced by $\{0010,\,\, 1101,\,\, 1110,\,\, 0110,\,\, 0001,\,\, 1001,\,\, 1010,\,\, 0101\}$ and a special 4-cycle is $C = <0100,\,\, 1011,\,\, 1100,\,\, 0011,\,\, 0100>$

\begin{lem}\label{l2}
	Let $L$ be a ladder-like graph on $2m$ vertices with $ m\geq 6$ as defined above. Then the subgraph  $ L-\{u_1v_1, u_4v_4\}$ of $L$ is  3-regular, 3-connected and 4-pancyclic.
\end{lem}
\begin{proof} Let $G = L-\{u_1v_1~,~ u_4v_4\}.$   Obviously,  $G$ is 3-regular.  It is easy to see  that deletion of any two edges from $G$ does not disconnect $G.$ Hence the edge-connectivity of $G$ is three.  Therefore, the vertex connectivity of $G$ is also three as  $G$ is a regular graph. Therefore $G$ is 3-connected. 
	
	We now prove that $G$ is pancyclic by constructing cycles of every length from 4 to $|V(G)|.$  By Lemma \ref{l1},  the ladder in $G$ formed by the paths  $<u_5, u_6, \dots, u_{m}>$ and $ <v_5, v_6, \dots, v_{m}>$ and the edges $u_iv_i ,$ for $ i = 5, 6,...,m,$ between them, is  bipancyclic. Therefore $G$ contains cycles of every even length from 4 to $2m - 8.$   The cycles  in $G$ of lengths $ 2m-6, 2m-4, 2m-2$ and $2m$ are $ <u_3,\dots, u_{m - 1}, v_{m - 1}, \dots v_3, u_3>, $ \\ $
  <u_3, \dots,  u_m, v_m, \dots,  v_3, u_3>,  <u_2, \dots  u_m, v_m, \dots, v_2, u_2>$ and \\
  $ <u_3, \dots u_m, u_1, u_2, v_2, v_1, v_m, \dots, v_3, u_3>,$ respectively. 
  
	We now construct cycles of odd lengths using edges $u_1v_4$ and $u_4v_1.$  A cycle of length 5 is  $<u_2, u_1, v_4, v_3, v_2, u_2>.$ For  $5 \leq i \leq m,$  $<u_4, v_1, v_2, v_3, ..., v_i, u_i, u_{i-1}, ..., u_4 >$ is a cycle of length $2i - 3.$  Thus we have constructed the cycles of all odd lengths from 5 to $2m-3.$ Finally, a cycle of length $2m-1$ is  $ <u_3, u_4, \dots, u_m, u_1, v_4, v_5, \dots, v_m, v_1, v_2, v_3, u_3>.$ Hence $G$ is pancyclic.
\end{proof}
\begin{cor}
A ladder-like graph is 3-connected and 4-pancyclic.
\end{cor}

	
We continue using the notations given in the Definition \ref{d1}. We now give two types of  constructions to get a new ladder-like graph from the two copies of a ladder-like graph.

%
%

\begin{lem}\label{l3}
	Let $n \geq 5.$ Consider the augmented cube $AQ_n$ as $AQ_n = AQ_{n-1}^0 \cup AQ_{n-1}^1 \cup E_n^h \cup E_n^c.$  Let $l_1$ be a spanning ladder-like subgraph of $ AQ_{n-1}^0$ and $l_1'$ be the corresponding subgraph of $AQ_{n-1}^1.$ Let $l_2$ be a spanning ladder-like subgraph of $ AQ_{n-1}^0$ with special 4-cycle and $l_2'$ be the corresponding subgraph of $AQ_{n-1}^1.$ Then we can construct 
	
	(i) a spanning ladder-like subgraph  $L_1$ of $ AQ_{n}$ from $l_1,$  $ l_1'$ and using four edges of $E_n^h;$ 
	
	(ii) a spanning ladder-like subgraph $L_2$ of $ AQ_{n}$ with special 4-cycle, from $l_2,$ $ l_2'$ and using four edges of $E_n^c.$ 
\end{lem}

\begin{proof}
Let $m = 2^{n-2}.$ Suppose the vertices of $l_1$ are $u_i$ and $v_i, $ $1 \leq i \leq  m,$  as in Definition \ref{d1}. Then $l_1'$ is a spanning ladder-like subgraph of $AQ_{n-1}^1$ corresponding to $l_1$ with vertices $u_i'$ corresponding to $u_i$ and $v_i'$ corresponding to $v_i$ for $1\leq i \leq m.$
Choose $s$ from the set $\{5, 6,\dots , m-1\}$ and fix it. Then $F_1 = \{u_su'_s, v_sv'_s, u_{s+1}u'_{s+1}, v_{s+1}v'_{s+1} \} \subset E_n^h.$ 

We obtain the graph $L_1$ from $l_1, l_1'$ and $F_1$ as follows.

\noindent	
	Let
	$L_1 = (l_1 -\{u_s u_{s+1},\; v_s v_{s+1}\})~\cup~( l_1' -\{u'_su'_{s+1},\; v'_sv'_{s+1},\; u'_1v'_4,\; u'_4v'_1\})~\cup~F_1$ (see Figure 3(a)).\\

		\begin{center}
			\vspace*{-0.19 in}
		\begin{tikzpicture}[scale=0.8]
		
		\draw [fill=black] (2,0) circle  (.08) node [left]  at (2,0) {$u_{m}$};
		\draw [fill=black] (2,1) circle  (.08) node [left]  at (2,1) {$u_{s+1}$}; 
		\draw [fill=black] (2,2) circle  (.08) node [left]  at (2,2) {$u_{s}$}; 
		\draw [fill=black] (2,3) circle  (.08) node [left]  at (2,3) {$u_{5}$}; 
		\draw [fill=black] (2,4) circle  (.08) node [left]  at (1.9,4) {$u_{4}$}; 
		\draw [fill=black] (2,5) circle  (.08) node [left]  at (2,5) {$u_3$}; 
		\draw [fill=black] (2,6) circle  (.08) node [left]  at (2,6) {$u_2$}; 
		\draw [fill=black] (2,7) circle  (.08) node [left]  at (2,7) {$u_1$};

		\draw [fill=black] (5,0) circle  (.08)  node [right]  at (5,0) {$v_{m}$};
		\draw [fill=black] (5,1) circle  (.08)  node [right]  at (5,1) {$v_{s+1}$};
		\draw [fill=black] (5,2) circle  (.08)  node [right]  at (5.3,1.9) {$v_{s}$};
		\draw [fill=black] (5,3) circle  (.08)  node [right]  at (5,2.9) {$v_5$};
		\draw [fill=black] (5,4) circle  (.08)  node [right]  at (5,4) {$v_{4}$};
		\draw [fill=black] (5,5) circle  (.08)  node [right]  at (5,5) {$v_3$};
		\draw [fill=black] (5,6) circle  (.08)  node [right]  at (5,6) {$v_2$};
		\draw [fill=black] (5,7) circle  (.08)  node [right]  at (5,7) {$v_1$};

		\draw [style= dotted, thick] (2,0)--(2,1);
		\draw [style= dotted, thick] (5,0)--(5,1);
	
		\draw [style= dotted, thick] (2,2)--(2,3);
		\draw [style= dotted, thick] (5,2)--(5,3);
		\draw [style= thin ] (2,3)--(2,7);
		\draw [style= thin] (5,3)--(5,7);
		
		\draw [style= thin] (2,0)--(5,0);
		\draw [style= thin] (2,1)--(5,1);
		\draw [style= thin] (2,2)--(5,2);
		\draw [style= thin] (2,3)--(5,3);
		\draw [style= thin] (2,4)--(5,4);
		\draw [style= thin] (2,5)--(5,5);
		\draw [style= thin] (2,6)--(5,6);
		\draw [style= thin] (2,7)--(5,7);

		\draw [style= thin] (2,0)..controls(1, 4)..(2,7);
		\draw [style= thin] (5,0)..controls(6, 4)..(5,7);
		
		\draw [style= thin] (2,7)--(5,4);
		\draw [style= thin] (2,4)--(5,7);

		
		\draw [fill=black] (8,0) circle  (.08) node [left]  at (8,0) {$u'_{m}$};
		\draw [fill=black] (8,1) circle  (.08) node [left]  at (8,1) {$u'_{s+1}$}; 
		\draw [fill=black] (8,2) circle  (.08) node [left]  at (7.85,1.9) {$u'_{s}$}; 
		\draw [fill=black] (8,3) circle  (.08) node [left]  at (8.05,2.9) {$u'_{5}$}; 
		\draw [fill=black] (8,4) circle  (.08) node [left]  at (7.85,4) {$u'_{4}$}; 
		\draw [fill=black] (8,5) circle  (.08) node [left]  at (8,5) {$u'_3$}; 
		\draw [fill=black] (8,6) circle  (.08) node [left]  at (8,6) {$u'_2$}; 
		\draw [fill=black] (8,7) circle  (.08) node [left]  at (8,7) {$u'_1$}; 

		\draw [fill=black] (11,0) circle  (.08)  node [right]  at (11,0) {$v'_{m}$};
		\draw [fill=black] (11,1) circle  (.08)  node [right]  at (11,1) {$v'_{s+1}$};
		\draw [fill=black] (11,2) circle  (.08)  node [right]  at (11.19,1.9) {$v'_{s}$};
		\draw [fill=black] (11,3) circle  (.08)  node [right]  at (11,3) {$v'_5$};
		\draw [fill=black] (11,4) circle  (.08)  node [right]  at (11,4) {$v'_{4}$};
		\draw [fill=black] (11,5) circle  (.08)  node [right]  at (11,5) {$v'_3$};
		\draw [fill=black] (11,6) circle  (.08)  node [right]  at (11,6) {$v'_2$};
		\draw [fill=black] (11,7) circle  (.08)  node [right]  at (11,7) {$v'_1$};

		\draw [style= dotted, thick] (8,0)--(8,1);
		\draw [style= dotted, thick] (11,0)--(11,1);
		\draw [style= dotted, thick] (8,2)--(8,3);
		\draw [style= dotted, thick] (11,2)--(11,3);
		\draw [style= thin ] (8,3)--(8,7);
		\draw [style= thin] (11,3)--(11,7);
		\draw [style= thin] (8,0)--(11,0);
		\draw [style= thin] (8,1)--(11,1);
		\draw [style= thin] (8,2)--(11,2);
		\draw [style= thin] (8,3)--(11,3);
		\draw [style= thin] (8,4)--(11,4);
		\draw [style= thin] (8,5)--(11,5);
		\draw [style= thin] (8,6)--(11,6);
		\draw [style= thin] (8,7)--(11,7);
		
		\draw [style= thin] (8,0)..controls(7, 4)..(8,7);
		\draw [style= thin] (11,0)..controls(12, 4)..(11,7);
		
		
		\draw [style= thin] (2,1)..controls(5, 1.8)..(8,1);
		\draw [style= thin] (2,2)..controls(5, 2.8)..(8,2);
		\draw [style= thin] (5,1)..controls(8, 1.8)..(11,1);
		\draw [style= thin] (5,2)..controls(8, 2.8)..(11,2);
		
		\end{tikzpicture}
		
		\hspace{0.001in} $l_1$ \hspace{1.9in } $l_1'$
		
		\hspace{0.3in} {Figure 3(a): Construction of $L_1$}
		\end{center}
		
		
		We show that the graph $L_1$ constructed above is a ladder-like graph on $4m = 2^n$ vertices. Observe that $ V(L_1) = \displaystyle \cup_{i=1}^{m}\{\;u_i, \;u_i', \;v_i,\; v_i'\;\}.$ Relabel the vertices of $L_1$ as follows. Let 
			
			\begin{displaymath}	
			\begin{tabular}{lllll}
			$x_i = \begin{cases}
			u_i & {\rm if}~~1 \leq i \leq s\\
			u'_{2s-i+1} & {\rm if}~~s+1 \leq i \leq 2s\\
			u'_{m+2s-i+1} & {\rm if}~~2s+1 \leq i \leq m+s\\
			u_{i-m} & {\rm if}~~m+s+1 \leq i \leq 2m\\
			\end{cases} $
			& & & & ${\rm and }~~~y_i = \begin{cases}
			v_i & {\rm if}~~1 \leq i \leq s\\
			v'_{2s-i+1} & {\rm if}~~s+1 \leq i \leq 2s\\
			v'_{m+2s-i+1} & {\rm if}~~2s+1 \leq i \leq m+s\\
			v_{i-m} & {\rm if}~~m+s+1 \leq i \leq 2m\\
			\end{cases} $\\
			
			\end{tabular}
			\end{displaymath}
			
			Then $ Z_1 = <x_1, x_2, \dots, x_{2m}, x_1>$  and  $ Z_2 = <y_1, y_2, \dots, y_{2m}, y_1>$ are cycles in $L_1.$ Observe that $L_1 = Z_1 \;\cup\; Z_2 \;\cup \;\{\;x_i y_i \colon i = 1, 2,\dots, 2m\;\} \;\cup \; \{\;x_1y_4,\; x_4y_1\;\}.$ Further, subgraph $B_8$ of $l_1$ which is induced by $\{u_1, u_2, u_3, u_4, v'_1, v'_2,  v'_3, v'_4\}$ is now induced by vertices $\{x_1, x_2, x_3, x_4, y_1, y_2, y_3, y_4\}.$ This shows that $L_1$ is a ladder-like graph on $4m$ vertices. 
			
			Now we construct a ladder-like subgraph $L_2$ with a speial 4-cycle from $l_2$ and $l'_2.$ Suppose the vertices of $l_2$ are $a_i$ and $b_i, $ $1 \leq i \leq  m,$  as in Definition \ref{d1}. Then $l_2'$ is a spanning ladder-like subgraph of $AQ_{n-1}^1$ corresponding to $l_2$ with vertices $a_i'$ corresponding to $a_i$ and $b_i'$ corresponding to $b_i$ for $1\leq i \leq m.$
			Let $C = < a_t, a_{t+1}, b_{t+1}, b_{t}, a_{t}>$ where $t \geq 5$ and $ a_{t+1} = a_t^c,\,\, b_{t+1} = b_t^c,$ be a special 4-cycle of $l_2.$ Suppose, $C' = < a'_t, a'_{t+1}, b'_{t+1}, b'_{t}, a'_{t}>$ is the corresponding special 4-cycle of $l_2'.$ In $AQ^0_{n-1,}$ the complement of $a_t$ is $ a_{t+1}$ and the complement of $b_t$ is $b_{t+1}.$ Similarly $a'_{t+1}$ is the complement of $a'_t$ and $b'_{t+1}$ is the complement of $b'_t,$ in $AQ^1_{n-1}.$ Thus, in $AQ_n, \, a'_{t+1}$ is the complement of $a_t, \, a'_t$ is the complement of $a_{t+1}, \, b'_{t+1}$ is the complement of $b_t$ and $b'_t$ is the complement of $b_{t+1}.$ Then $F_2 = \{a_ta'_{t+1}, b_tb'_{t+1}, a_{t+1}a'_{t}, b_{t+1}b'_{t} \} \subset E_n^c.$  
			Now, we obtain the ladder-like graph $L_2$ with a special 4-cycle from $l_2, l_2'$ and using $F_2$ as follows.
			
		$L_2 = (l_2 -\{a_t a_{t+1},\; b_t b_{t+1}\})~\cup~( l_2' -\{a'_ta'_{t+1},\; b'_tb'_{t+1},\; a'_1b'_4,\; a'_4b'_1\})~\cup~F_2.$ (see Figure 3(b)).
		
		\begin{center} 
		\begin{tikzpicture}[scale=0.8]
		\draw [fill=black] (2,0) circle  (.08) node [left]  at (2,0) {$a_{m}$};
		\draw [fill=black] (2,1) circle  (.08) node [left]  at (1.9,1) {$a_{t+1}$}; 
		\draw [fill=black] (2,2) circle  (.08) node [left]  at (2.1,2) {$a_{t}$}; 
		\draw [fill=black] (2,3) circle  (.08) node [left]  at (2,3) {$a_{5}$}; 
		\draw [fill=black] (2,4) circle  (.08) node [left]  at (1.9,4) {$a_{4}$}; 
		\draw [fill=black] (2,5) circle  (.08) node [left]  at (2,5) {$a_3$}; 
		\draw [fill=black] (2,6) circle  (.08) node [left]  at (2,6) {$a_2$}; 
		\draw [fill=black] (2,7) circle  (.08) node [left]  at (2,7) {$a_1$}; 

		\draw [fill=black] (5,0) circle  (.08)  node [right]  at (5,0) {$b_{m}$};
		\draw [fill=black] (5,1) circle  (.08)  node [right]  at (5.2,1.1) {$b_{t+1}$};
		\draw [fill=black] (5,2) circle  (.08)  node [right]  at (4.9,1.8) {$b_{t}$};
		\draw [fill=black] (5,3) circle  (.08)  node [right]  at (4.95,2.8) {$b_5$};
		\draw [fill=black] (5,4) circle  (.08)  node [right]  at (5,4) {$b_{4}$};
		\draw [fill=black] (5,5) circle  (.08)  node [right]  at (5,5) {$b_3$};
		\draw [fill=black] (5,6) circle  (.08)  node [right]  at (5,6) {$b_2$};
		\draw [fill=black] (5,7) circle  (.08)  node [right]  at (5,7) {$b_1$};

		\draw [style= dotted] (2,0)--(2,1);
		\draw [style= dotted] (5,0)--(5,1);
		\draw [style= dotted] (2,2)--(2,3);
		\draw [style= dotted] (5,2)--(5,3);
		\draw [style= thin ] (2,3)--(2,7);
		\draw [style= thin] (5,3)--(5,7);
		\draw [style= thin] (2,0)--(5,0);
		\draw [style= thin] (2,1)--(5,1);
		\draw [style= thin] (2,2)--(5,2);
		\draw [style= thin] (2,3)--(5,3);
		\draw [style= thin] (2,4)--(5,4);
		\draw [style= thin] (2,5)--(5,5);
		\draw [style= thin] (2,6)--(5,6);
		\draw [style= thin] (2,7)--(5,7);
		
		\draw [style= thin] (2,0)..controls(1, 4)..(2,7);
		\draw [style= thin] (5,0)..controls(6, 4)..(5,7);
		
		\draw [style= thin] (2,7)--(5,4);
		\draw [style= thin] (2,4)--(5,7);

		
		\draw [fill=black] (8,0) circle  (.08) node [left]  at (8,0) {$a'_{m}$};
		\draw [fill=black] (8,1) circle  (.08) node [left]  at (7.9,1) {$a'_{t+1}$}; 
		\draw [fill=black] (8,2) circle  (.08) node [left]  at (8.2,1.6) {$a'_{t}$}; 
		\draw [fill=black] (8,3) circle  (.08) node [left]  at (8.05,2.9) {$a'_{5}$}; 
		\draw [fill=black] (8,4) circle  (.08) node [left]  at (7.85,4) {$a'_{4}$}; 
		\draw [fill=black] (8,5) circle  (.08) node [left]  at (8,5) {$a'_3$}; 
		\draw [fill=black] (8,6) circle  (.08) node [left]  at (8,6) {$a'_2$}; 
		\draw [fill=black] (8,7) circle  (.08) node [left]  at (8,7) {$a'_1$}; 

		\draw [fill=black] (11,0) circle  (.08)  node [right]  at (11,0) {$b'_{m}$};
		\draw [fill=black] (11,1) circle  (.08)  node [right]  at (11.1,1) {$b'_{t+1}$};
		\draw [fill=black] (11,2) circle  (.08)  node [right]  at (10.95,1.9) {$b'_{t}$};
		\draw [fill=black] (11,3) circle  (.08)  node [right]  at (11,3) {$b'_5$};
		\draw [fill=black] (11,4) circle  (.08)  node [right]  at (11,4) {$b'_{4}$};
		\draw [fill=black] (11,5) circle  (.08)  node [right]  at (11,5) {$b'_3$};
		\draw [fill=black] (11,6) circle  (.08)  node [right]  at (11,6) {$b'_2$};
		\draw [fill=black] (11,7) circle  (.08)  node [right]  at (11,7) {$b'_1$};

		\draw [style= dotted] (8,0)--(8,1);
		\draw [style= dotted] (11,0)--(11,1);
		\draw [style= dotted] (8,2)--(8,3);
		\draw [style= dotted] (11,2)--(11,3);
		\draw [style= thin ] (8,3)--(8,7);
		\draw [style= thin] (11,3)--(11,7);
		\draw [style= thin] (8,0)--(11,0);
		\draw [style= thin] (8,1)--(11,1);
		\draw [style= thin] (8,2)--(11,2);
		\draw [style= thin] (8,3)--(11,3);
		\draw [style= thin] (8,4)--(11,4);
		\draw [style= thin] (8,5)--(11,5);
		\draw [style= thin] (8,6)--(11,6);
		\draw [style= thin] (8,7)--(11,7);
		
		\draw [style= thin] (8,0)..controls(7, 4)..(8,7);
		\draw [style= thin] (11,0)..controls(12, 4)..(11,7);
		
		
		\draw [style= thin] (2,1)..controls(5.5, 0.1)..(8,2);
		\draw [style= thin] (2,2)..controls(5, 2.8)..(8,1);
		\draw [style= thin] (5,1)..controls(8.5, 0.1)..(11,2);
		\draw [style= thin] (5,2)..controls(8, 2.8)..(11,1);
		
		\end{tikzpicture}
		
		\hspace{0.001in} $l_2$ \hspace{1.9in } $l_2'$
			
		\hspace{0.3in} {Figure 3(b): Construction of $L_2$}
		\end{center}

To prove that $L_2$  is a ladder-like graph, we relabel again the vertices of $L_2$ as follows. Let
	
	\begin{displaymath}	
	\begin{tabular}{lllll}
	$p_i =\begin{cases}
	a_i & {\rm if}~~1 \leq i \leq t\\
	a'_i & {\rm if}~~t+1 \leq i \leq m\\
	a'_{i-m} & {\rm if}~~m+1 \leq i \leq m+t\\
	a_{i-m} & {\rm if}~~m+t+1 \leq i \leq 2m\\
	\end{cases} $
	& & & & ${\rm and}~~~q_i =\begin{cases}
	b_i & {\rm if}~~1 \leq i \leq t\\
	b'_i & {\rm if}~~t+1 \leq i \leq m\\
	b'_{i-m} & {\rm if}~~m+1 \leq i \leq m+t\\
	b_{i-m} & {\rm if}~~m+t+1 \leq i \leq 2m\\
	\end{cases} $
	
	\end{tabular}
	\end{displaymath}
	
	Clearly, $L_2 = \;Z_3 \;\cup\; Z_4 \; \cup \; \{\; p_i \; q_i \;\colon \; i =1, 2, \dots, 2m\; \} \;\cup \;\{\;p_1q_4,\; q_1p_4\;\},$ where	
	$ Z_3 = <p_{1}, p_{2}, \dots, p_{2m}, p_1 >$ and  $Z_4 =  <q_1, q_2, \dots, q_{2m}, q_1>$ are cycles in $L_2.$ Also, the subgraph of $L_2$  induced by the vertex set $\{ p_1, p_2, p_3, p_4, q_1, q_2, q_3, q_4\}$ is $B_8.$\\
	Due to renaming we have, $p_t = a_t,\, \, p_{t+1} = a'_{t+1}, \,\, q_{t+1} = b'_{t+1}$ and $q_t = b_t.$  Hence, 
	
	$<p_t, \,p_{t+1},\, q_{t+1}, \,q_t>$ is a special 4-cycle in $L_2.$ Thus $L_2$ is a ladder-like graph on $4m$ vertices.
	\end{proof}

\begin{cor}
	For $n \geq 4,$ the augmented cube $AQ_n$ contains a spanning ladder-like subgraph and a spanning ladder-like subgraph with a special 4-cycle. 
\end{cor}



\begin{lem}\label{l4}
	Let $H_1$ and $H_2$ be vertex disjoint k-connected graphs with $V(H_1) = \{ u_1, u_2, \dots, u_n\}$ and $V(H_2) = \{v_1, v_2, \dots, v_n\}$ and let $G$ be the graph obtained from $H_1 \cup H_2$ by adding the edges $u_i v_i$ for $i = 1, 2, \dots, n.$ Then $G$ is $(k+1)$-connected.
\end{lem}
\begin{proof}
	Let $S \subset V(G)$ with $|S| = k.$ Suppose $S \subseteq V(H_1).$ Then each component of $H_1 - S$ is connected to $H_2$ by an edge of type $u_i v_i$ for some $i$. Hence, $G - S$ is connected. Similarly, $G - S$ is connected if $S \subseteq V(H_2).$ Suppose $S = S_1 \cup S_2,$ where $S_1 \subset V(H_1)$ and $S_2 \subset V(H_2).$ Then $|S_1| < k$ and $|S_2| < k.$ Since $H_i$ is k-connected, $H_i - S_i$ is connected for each $i \in \{1, 2\}.$ Further, there is an edge between $H_1 - S_1$ and $H_2 - S_2.$ Hence $G - S$ is connected. Therefore $G$ is $(k+1)$-connetced.
\end{proof}

\section{ \textbf{Case $n_1 = 2$}	}
In this section, we prove Theorem \ref{2} for the special case $n_1 = 2.$ First we prove it for  $n = 4.$ 
\begin{lem}\label{3.1}
	 There exists a hamiltonian cycle $C$ in $AQ_4$ such that $AQ_4 - E(C)$ is a 5-regular, 5-connected graph containing a spanning ladder-like subgraph of $AQ_4.$
\end{lem}
\begin{proof}
Consider the hamiltonian cycle $C = <0111, 1000, 1100, 0100, 0000, 1111, 1011, 0011,\\ 0010, 1010, 1110, 0001, 0101, 1101, 1001, 0110, 0111>$ of $AQ_4.$ This cycle is denoted by dashed edges in Figure 4(a). Let $H = AQ_4 - E(C).$ Then, $H$ is a spanning, 5-regular subgraph of $AQ_4.$ It is easy to see that $H$ is union of two 4-connected graphs, each on eight vertices and one perfect matching between them (see Figure 4(c)).

	\begin{tikzpicture}[scale=0.7]
	\vspace*{-7.2cm}
	\hspace*{-2.5cm}


\draw [fill=black] (0,0) circle  (.08)  node [left]  at (0,0.3) {$\textbf{0110}$};
\draw [fill=black] (3,0) circle  (.08)  node [left]  at (3,0.3) {$\textbf{0111}$};
\draw [fill=black] (5,0) circle  (.08)  node [left]  at (5,0.3) {$\textbf{1110}$};
\draw [fill=black] (8,0) circle  (.08)  node [left]  at (8,0.3) {$\textbf{1111}$};

\draw [fill=black] (0,3) circle  (.08)  node [left]  at (0,3.3) {$\textbf{0100}$};
\draw [fill=black] (3,3) circle  (.08)  node [left]  at (3,3.3) {$\textbf{0101}$};
\draw [fill=black] (5,3) circle  (.08)  node [left]  at (5,3.3) {$\textbf{1100}$};
\draw [fill=black] (8,3) circle  (.08)  node [left]  at (8,3.3) {$\textbf{1101}$};

\draw [fill=black] (0,5) circle  (.08)  node [left]  at (0,5.3) {$\textbf{0010}$};
\draw [fill=black] (3,5) circle  (.08)  node [left]  at (3,5.3) {$\textbf{0011}$};
\draw [fill=black] (5,5) circle  (.08)  node [left]  at (5,5.3) {$\textbf{1010}$};
\draw [fill=black] (8,5) circle  (.08)  node [left]  at (8,5.3) {$\textbf{1011}$};

\draw [fill=black] (0,8) circle  (.08)  node [left]  at (0,8.3) {$\textbf{0000}$};
\draw [fill=black] (3,8) circle  (.08)  node [left]  at (3,8.3) {$\textbf{0001}$};
\draw [fill=black] (5,8) circle  (.08)  node [left]  at (5,8.3) {$\textbf{1000}$};
\draw [fill=black] (8,8) circle  (.08)  node [left]  at (8,8.3) {$\textbf{1001}$};

\draw [style= dashed] ( 0,0 ) -- ( 3,0 );
\draw [style= thin] ( 0,0 ) -- ( 0,3 );
\draw [style= thin] ( 0,0 ) -- ( 3,3 );
\draw [style= thin] (0,0)..controls(-0.7, 2.5)..(0,5);
\draw [style= thin] ( 0,0 ) -- ( 3,8 );
\draw [style= thin] ( 0,0 ) ..controls(2.5, -0.7).. ( 5,0 );
\draw [style= dashed] (0,0) to [out=180,in=-150] (-1,9) to [out=30,in=160] (8,8);
\draw [style= thin] ( 0,3 ) -- ( 3,0 );
\draw [style= thin] ( 0,3 ) -- ( 3,3 );
\draw [style= thin] ( 0,3 ) -- ( 3,5 );
\draw [style= dashed] ( 0,3 ) ..controls(-0.7, 5.5)..( 0,8 );
\draw [style= thin] ( 0,3 ) -- ( 8,5 );
\draw [style= dashed] ( 0,3 ) ..controls(2.5, 4).. ( 5,3 );
\draw [style= dashed] ( 0,5 ) -- ( 3,5 );
\draw [style= thin] ( 0,5 ) -- ( 0,8 );
\draw [style= thin] ( 0,5 ) -- ( 3,8 );
\draw [style= dashed] ( 0,5 ) ..controls(2.5, 6).. ( 5,5 );
\draw [style= thin] ( 0,5 ) -- ( 8,3 );
\draw [style= thin] ( 0,5 ) -- ( 3,3 );
\draw [style= thin] ( 0,8 ) -- ( 3,5 );
\draw [style= thin] ( 0,8 ) -- ( 3,8 );
\draw [style= thin] ( 0,8 ) -- ( 3,0 );
\draw [style= thin] ( 0,8 ) ..controls(2.5, 9).. ( 5,8 );
\draw [style= dashed] (0,8) to [out=-160,in=140] (-1,-0.5) to [out=-30,in=-160] (8,0);
\draw [style= thin] ( 3,0 ) -- ( 3,3 );
\draw [style= dashed] ( 3,0 ) -- ( 5,8 );
\draw [style= thin] ( 3,0 ) ..controls(3.5, 3).. ( 3,5 );
\draw [style= thin] ( 3,0 ) ..controls(5.5, -0.7).. ( 8,0 );
\draw [style= thin] ( 3,3 ) -- ( 5,5 );
\draw [style= dashed] ( 3,3 ) ..controls(5.5, 4).. ( 8,3 );
\draw [style= dashed] ( 3,3 ) ..controls(3.5, 5.5).. ( 3,8 );
\draw [style= thin] ( 3,5 ) -- ( 3,8 );
\draw [style= dashed] ( 3,5 ) ..controls(5.5, 6).. ( 8,5 );
\draw [style= thin] ( 3,5 ) -- ( 5,3 );
\draw [style= thin] ( 3,8 ) ..controls(5.5, 9).. ( 8,8 );
\draw [style= dashed] ( 3,8 ) -- ( 5,0 );
\draw [style= thin] ( 5,0 ) -- ( 8,0 );
\draw [style= thin] ( 5,0 ) -- ( 8,3 );
\draw [style= thin] ( 5,0 ) -- ( 5,3 );
\draw [style= thin] ( 5,0 ) -- ( 8,8 );
\draw [style= dashed] ( 5,0 ) ..controls(4.5, 3.5).. ( 5,5 );
\draw [style= thin] ( 5,3 ) -- ( 8,0 );
\draw [style= thin] ( 5,3 ) -- ( 8,5 );
\draw [style= thin] ( 5,3 ) -- ( 8,3 );
\draw [style= dashed] ( 5,3 ) ..controls(4.5, 6).. ( 5,8 );
\draw [style= thin] ( 5,5 ) -- ( 5,8 );
\draw [style= thin] ( 5,5 ) -- ( 8,5 );
\draw [style= thin] ( 5,5 ) -- ( 8,8 );
\draw [style= thin] ( 5,5 ) -- ( 8,3 );
\draw [style= thin] ( 5,8 ) -- ( 8,8 );
\draw [style= thin] ( 5,8 ) -- ( 8,5 );
\draw [style= thin] ( 5,8 ) -- ( 8,0 );
\draw [style= thin] ( 8,0 ) -- ( 8,3 );
\draw [style= dashed] ( 8,0 ) ..controls(8.5, 3).. ( 8,5 );
\draw [style= dashed] ( 8,3 ) ..controls(8.5, 5.5).. ( 8,8 );
\draw [style= thin] ( 8,5 ) -- ( 8,8 );



\hspace*{-1.0cm}
\draw [style= thin, color= gray] (12,0)--(15,3);
\draw [style= thin, color= gray] (12,3)--(15,0);
\draw [style= thin, color= gray] (12,2)--(15,6);
\draw [style= thin, color= gray] (12,1)--(15,5);
\draw [style= thin, color= gray] (12,5)--(15,2);
\draw [style= thin, color= gray] (12,6)--(15,1);

\draw [style= thin, color= gray] (12,5)--(15,1);
\draw [style= thin, color= gray] (12,6)--(15,2);
\draw [style= thin, color= gray] (12,1)--(15,6);
\draw [style= thin, color= gray] (12,2)--(15,5);

\draw [style= thin, color= gray] (12,7)..controls(11.5, 5.5)..(12,4);
\draw [style= thin, color= gray] (15,0)..controls(15.5, 1.5)..(15,3);
\draw [style= thin, color= gray] (12,3)--(15,7);
\draw [style= thin, color= gray] (12,0)--(15,4);

\draw [fill=black] (12,0) circle  (.08)  node [left]  at (12,0) {${ 0000}$};
\draw [fill=black] (12,1) circle  (.08)  node [left]  at (11.9,1) {$1000$};
\draw [fill=black] (12,2) circle  (.08)  node [left]  at (11.9,2) {$1011$};
\draw [fill=black] (12,3) circle  (.08)  node [left]  at (12,3) {$0100$};
\draw [fill=black] (12,4) circle  (.08)  node [left]  at (12,4) {$0110$};
\draw [fill=black] (12,5) circle  (.08)  node [left]  at (11.9,5) {$1110$};
\draw [fill=black] (12,6) circle  (.08)  node [left]  at (12,6) {$1101$};
\draw [fill=black] (12,7) circle  (.08)  node [left]  at (12,7) {$0010$};

\draw [fill=black] (15,0) circle  (.08)  node [right]  at (15,0) {$0111$};
\draw [fill=black] (15,1) circle  (.08)  node [right]  at (15.3,1) {$1111$};
\draw [fill=black] (15,2) circle  (.08)  node [right]  at (15.2,2) {$1100$};
\draw [fill=black] (15,3) circle  (.08)  node [right]  at (15,3) {$0011$};
\draw [fill=black] (15,4) circle  (.08)  node [right]  at (15,4) {$0001$};
\draw [fill=black] (15,5) circle  (.08)  node [right]  at (15,5) {$1001$};
\draw [fill=black] (15,6) circle  (.08)  node [right]  at (15.1,6) {$1010$};
\draw [fill=black] (15,7) circle  (.08)  node [right]  at (15,7) {$0101$};

\draw [style= thick] (12,0)--(12,1);      
\draw [style= thick] (15,0)--(15,1);
\draw [style= thick] (12,1)--(12,2);
\draw [style= thick] (15,1)--(15,2);
\draw [style= thick] (12,2)--(12,7);
\draw [style= thick] (15,2)--(15,7);
\draw [style= thick] (12,0)--(15,0);
\draw [style= thick] (12,1)--(15,1);
\draw [style= thick] (12,2)--(15,2);
\draw [style= thick] (12,3)--(15,3);
\draw [style= thick] (12,4)--(15,4);
\draw [style= thick] (12,5)--(15,5);
\draw [style= thick] (12,6)--(15,6);
\draw [style= thick] (12,7)--(15,7);

\draw [style= thick] (12,0)..controls(10, 4)..(12,7);
\draw [style= thick] (15,0)..controls(17, 4)..(15,7);

\draw [style= thick] (12,7)--(15,4);
\draw [style= thick] (12,4)--(15,7);

\draw [fill=black] ( 19, 0) circle  (.08)  node [left]  at (19, 0) {${1000}$};
\draw [fill=black] (19, 1 ) circle  (.08)  node [left]  at (19, 1) {${1011}$};
\draw [fill=black] ( 19, 2) circle  (.08)  node [left]  at (19, 2) {${1110}$};
\draw [fill=black] (19, 3 ) circle  (.08)  node [left]  at (19, 3) {${1101}$};
\draw [fill=black] ( 19, 5) circle  (.08)  node [left]  at (19, 5) {${0000}$};
\draw [fill=black] (19, 6 ) circle  (.08)  node [left]  at (19, 6) {${0100}$};
\draw [fill=black] ( 19, 7) circle  (.08)  node [left]  at (19, 7) {${0110}$};
\draw [fill=black] (19, 8 ) circle  (.08)  node [left]  at (19, 8) {${0010}$};

\draw [fill=black] ( 22, 0) circle  (.08)  node [right]  at (22, 0) {${1111}$};
\draw [fill=black] (22, 1 ) circle  (.08)  node [right]  at (22, 1) {${1100}$};
\draw [fill=black] ( 22, 2) circle  (.08)  node [right]  at (22, 2) {${1001}$};
\draw [fill=black] (22, 3 ) circle  (.08)  node [right]  at (22, 3) {${1010}$};
\draw [fill=black] ( 22, 5) circle  (.08)  node [right]  at (22, 5) {${0111}$};
\draw [fill=black] (22, 6 ) circle  (.08)  node [right]  at (22, 6) {${0011}$};
\draw [fill=black] ( 22, 7) circle  (.08)  node [right]  at (22, 7) {${0001}$};
\draw [fill=black] (22, 8 ) circle  (.08)  node [right]  at (22, 8) {${0101}$};

\draw [style= thin] (19 ,0 )--(19 ,1 );
\draw [style= thin] ( 19,0 )--( 22,0 );
\draw [style= thin] ( 19,0 )--( 22,3 );
\draw [style= thin] ( 19,0 )--( 22,2 );
\draw [style= thin] ( 19, 1)--( 22, 1);
\draw [style= thin] ( 19, 1)--(22 , 2);
\draw [style= thin] ( 19, 1)--( 22, 3);
\draw [style= thin] ( 19,2 )--( 22, 0 );
\draw [style= thin] ( 19,2 )--( 22, 1);
\draw [style= thin] ( 19,2)--( 22,2 );
\draw [style= thin] ( 19,2 )--( 19,3 );
\draw [style= thin] ( 19,3 )--( 22,3 );
\draw [style= thin] ( 19,3 )--( 22,0 );
\draw [style= thin] ( 19,3 )--( 22,1 );
\draw [style= thin] ( 22, 3)--(22 , 2);
\draw [style= thin] (22 , 0)--(22 , 1);

\draw [style= thin] (19, 5 )--( 22,5 );
\draw [style= thin] ( 19,5 )--( 22,6 );
\draw [style= thin] (19 ,5 )--(22 , 7);
\draw [style= thin] ( 19,5 ) ..controls(20, 6.5)..( 19,8 );
\draw [style= thin] ( 19,6 )--(22 ,5 );
\draw [style= thin] ( 19,6 )--( 22,6 );
\draw [style= thin] ( 19,6 )--( 22,8 );
\draw [style= thin] ( 19,6 )--( 19,7 );
\draw [style= thin] ( 19,7 )--( 22,7 );
\draw [style= thin] ( 19,7 )--( 22,8 );
\draw [style= thin] ( 19,7 )--( 19,8 );
\draw [style= thin] ( 19,8 )--( 22,7 );
\draw [style= thin] ( 19,8 )--( 22,8 );
\draw [style= thin] ( 22,5 )--( 22,6 );
\draw [style= thin] ( 22,6 )--( 22,7 );
\draw [style= thin] ( 22,5 )..controls(21, 6.5)..(22 ,8 );

\draw [style= thin] ( 19, 0  )..controls(17, 3 )..(19 , 5);
\draw [style= thin] ( 19, 1 )..controls(17, 4 )..(19 , 6);
\draw [style= thin] ( 19, 2 )..controls(17, 5)..(19 , 7);
\draw [style= thin] ( 19,  3)..controls(17, 6)..(19 , 8);

\draw [style= thin] ( 22, 0 )..controls(24, 3 )..(22 , 5);
\draw [style= thin] ( 22, 1 )..controls(24, 4)..(22 , 6);
\draw [style= thin] ( 22, 2 )..controls(24, 5 )..(22 , 7);
\draw [style= thin] ( 22, 3 )..controls(24, 6 )..(22 ,8 );

\end{tikzpicture}

\hspace{0.1in} { (a).  $AQ_4 = C \cup H$} \hspace{0.9in} $\,\,~~$ {(b). $ H = AQ_4 - E(C)$} \hspace{0.4in} {(c). $ H = AQ_4 - E(C)$}

\hspace{1 in} {Figure 4: Decomposition of $AQ_4$ for $n_1 = 2$}\\

Hence, by Lemma \ref{l4}, $H$ is $5$-connected. Further, $H$ contains a spanning ladder-like subgraph as shown in Figure 4(b) by dark lines. 
\end{proof}

We now prove Theorem \ref{2} for the spacial case $n_1 = 2.$
\begin{proposition} 		
For $ n \geq 4,$ there exists a hamiltonian cycle $C$ in $AQ_n$ such that $AQ_n - E(C)$ is a spanning, $(2n-3)$-regular, $(2n-3)$-connected subgraph of $AQ_n$ and it contains a spanning ladder-like subgraph.
\end{proposition}

\begin{proof}
We proceed by induction on $n.$ By Lemma \ref{3.1}, result is true for $n=4.$ Suppose $n \geq 5.$  Write $AQ_n$ as $AQ_n = AQ_{n-1}^0 \cup AQ_{n-1}^1 \cup E_n^h \cup E_n^c.$ By induction hypothesis, suppose the result is true for $AQ_{n-1}.$ Let $C_0$ be a hamiltonian cycle of $AQ_{n-1}^0$ such that $AQ_{n-1}^0 - E(C_0)$ is a spanning $(2n -5)$-regular and $(2n-5)$-connected containing a spanning ladder-like subgraph $L.$ Let $C_1$ be the corresponding hamiltonian cycle in $AQ_{n-1}^1.$ Then, $AQ_{n-1}^1 - E(C_1)$ is a spanning, $(2n -5)$-regular and $(2n-5)$-connected. Let $L'$ be the corresponding spanning ladder-like subgraph in $AQ_{n-1}^1 - E(C_1).$
 Let $a b \in E(C_0)$ and let $a'b' $ be the corresponding edge in $E(C_1).$ Let, $C = (~C_0-\{a b\}~) \cup (~C_1-\{a' b'\}~) \cup \{a a', b b'\}.$ Then $C$ is a hamiltonian cycle of $AQ_n.$ \\

Let $H = AQ_n - E(C).$ Then, $H = (AQ_{n-1}^0 -E(C_0)) \cup (AQ_{n-1}^1 -E(C_1)) \cup \{a b, a' b'\} \cup (E_n^h -\{ a a', b b'\}) \cup E_n^c.$ Obviously, $H$ is a spanning, $(2n - 3)$-regular subgraph of $AQ_n$. \\

We prove that $H$ is $(2n-3)$-connected. Let $S \subset V(AQ_n)$ with $|S| \leq 2n - 4.$ Write $S = S_0 \cup S_1$ with $S_0 = V(AQ^0_{n-1}) \cap S$ and $S_1 = V(AQ^1_{n-1}) \cap S.$ It is suffices to prove that $H - S$ is connected. Suppose $ S = S_0.$ As each component of $AQ_{n-1}^0 -E(C_0) - S$ is joined to $AQ_{n-1}^1 -E(C_1)$ by an edge of $E_n^h$ or $E_n^c, H - S$ is connected. Suppose $S_0 \neq \emptyset$ and $S_1 \neq \emptyset.$ If $|S_0| <2n - 5$ and $|S_1| <2n - 5,$ then  $AQ_{n-1}^i - E(C_i) - S_i,$ for $i = 0, 1$  are connected and are joined to each other by an edge. Thus $H - S$ is connected. If $|S_0| = 2n - 5$ and $|S_1| = 1.$ Then $\left(AQ_{n-1}^1 - E(C_1)\right) - S_1$ is connected. As each component of $\left(AQ_{n-1}^0 - E(C_0)\right)- S_0$ is connected to $\left(AQ_{n-1}^1 - E(C_1)\right) - S_1$ by at least one edge from $E_n^c, H - S$ is connected. Hence $H$ is $(2n - 3)$-connected.\\

Now we construct a spanning ladder-like subgraph of $H.$ By induction, there is a ladder-like subgraph $L$ in $AQ_{n-1}^0 - E(C_0)$ and corresponding subgraph $L'$ in $AQ_{n-1}^1 - E(C_1).$ By Lemma \ref{l3}(i), using $L$ and $L',$ we get a spanning ladder-like subgraph in $H.$
\end{proof}

 The case $n_1 = 2$ of Theorem \ref{2} follows from above proposition and Lemma \ref{l2}.

\section{\textbf{Case $n_1 = 3$}}

In this section, we prove Theorem \ref{2} for the special case $ n_1  = 3.$ First we prove the result for $n = 4.$ 

Recall from Definition \ref{d1} that a  ladder-like subgraph  of $AQ_n$ on $2m \geq 12$ vertices $ u_1, u_2, \dots, u_m$ and $v_1, v_2, \dots, v_m$ contains the edges  $ u_iv_i$ for $ i = 1, 2, \dots, m$ and  the two edges $u_1v_4,  u_4v_1.$ 
\begin{lem}
	The augmented cube $AQ_4$ contains a spanning ladder-like subgraph $L_1$ such that
	
	(i)  $ H = L_1 - \{ u_1v_1, u_4v_4 \}$ is 3-regular and 3-connected; 
	
	(ii)  $AQ_4 - E(H)$ is 4-regular, 4-connected containing a spanning ladder-like subgraph $L_2$ with a special 4-cycle, which avoids the edges $u_1v_1$ and $ u_4v_4.$
\end{lem}
\begin{proof}
 A decomposition of  $AQ_4$ into a spanning subgraph $H$ and its complement, where the edges of $H$ are denoted by dashed lines and the edges of the complement graph by solid lines, is shown in Figure 5(a). These two graphs are redrawn separately in Figure 5(b) and 5(c).	Let $u_1 = 1000,\,\, v_1 = 1001,\,\, u_4 = 1011,\,\, v_4 = 1010,$ and let $L_1 =  H \cup \{u_1v_1, u_4 v_4\}.$ Then, from Figure 5(b), it is clear that $L_1$ is a ladder-like spanning subgraph of $AQ_4.$ Clearly, $H$ is a 3-regular spanning  subgraph of $AQ_4.$ Also, by Lemma \ref{l2}, $H$ is 3-connected. This proves (i).
	\begin{center}
		\vspace*{-0.5cm}
		\begin{tikzpicture}[scale=0.7][h]
	\hspace*{-1.2cm}
	\draw [fill=black] (2.5,9) circle  (.08)  node [left]  at (2.5,9.3) {$\textbf{0110}$};
	\draw [fill=black] (5.5,9) circle  (.08)  node [left]  at (5.5,9.3) {$\textbf{0111}$};
	\draw [fill=black] (7.5,9) circle  (.08)  node [left]  at (7.5,9.3) {$\textbf{1110}$};
	\draw [fill=black] (10.5,9) circle (.08)  node [left]  at (10.5,9.3) {$\textbf{1111}$};
	
	\draw [fill=black] (2.5,12) circle  (.08)  node [left]  at (2.5,12.3) {$\textbf{0100}$};
	\draw [fill=black] (5.5,12) circle  (.08)  node [left]  at (5.5,12.3) {$\textbf{0101}$};
	\draw [fill=black] (7.5,12) circle  (.08)  node [left]  at (7.5,12.3) {$\textbf{1100}$};
	\draw [fill=black] (10.5,12) circle  (.08)  node [left]  at (10.5,12.3) {$\textbf{1101}$};
	
	\draw [fill=black] (2.5,14) circle  (.08)  node [left]  at (2.5,14.3) {$\textbf{0010}$};
	\draw [fill=black] (5.5,14) circle  (.08)  node [left]  at (5.5,14.3) {$\textbf{0011}$};
	\draw [fill=black] (7.5,14) circle  (.08)  node [left]  at (7.5,14.3) {$\textbf{1010}$};
	\draw [fill=black] (10.5,14) circle  (.08)  node [left]  at (10.5,14.3) {$\textbf{1011}$};
	
	\draw [fill=black] (2.5,17) circle  (.08)  node [left]  at (2.5,17.3) {$\textbf{0000}$};
	\draw [fill=black] (5.5,17) circle  (.08)  node [left]  at (5.5,17.3) {$\textbf{0001}$};
	\draw [fill=black] (7.5,17) circle  (.08)  node [left]  at (7.5,17.3) {$\textbf{1000}$};
	\draw [fill=black] (10.5,17) circle  (.08)  node [left]  at (10.5,17.3) {$\textbf{1001}$};
	
	\draw [style= dashed] ( 2.5,9) -- ( 5.5,9 );
	\draw [style= thin] ( 2.5,9) -- ( 2.5,12 );
	\draw [style= thin] ( 2.5,9) -- ( 5.5,12 );
	\draw [style= dashed] ( 2.5,9)..controls(1.7, 11.5)..( 2.5,14 );
	\draw [style= thin] ( 2.5,9) -- ( 5.5,17 );
	\draw [style= thin] ( 2.5,9) ..controls(4.5, 8.3).. ( 7.5,9 );
	\draw [style= dashed] ( 2.5,9) to [out=180,in=-150] (1,18) to [out=30,in=160] ( 10.5,17 );
	\draw [style= thin] ( 2.5,12 ) -- ( 5.5,9 );
	\draw [style= dashed] ( 2.5,12 ) -- ( 5.5,12 );
	\draw [style= thin] ( 2.5,12 ) -- ( 5.5,14 );
	\draw [style= dashed] ( 2.5,12 ) ..controls(1.7, 14.5)..( 2.5,17 );
	\draw [style= thin] ( 2.5,12 ) -- ( 10.5,14 );
	\draw [style= dashed] ( 2.5,12 ) ..controls(4.5, 13).. ( 7.5,12 );
	\draw [style= dashed] ( 2.5,14 ) -- ( 5.5,14 );
	\draw [style= thin] ( 2.5,14 ) -- ( 2.5,17 );
	\draw [style= thin] ( 2.5,14 ) -- ( 5.5,17 );
	\draw [style= dashed] ( 2.5,14 ) ..controls(4.5, 15).. ( 7.5,14 );
	\draw [style= thin] ( 2.5,14 ) -- ( 10.5,12 );
	\draw [style= thin] ( 2.5,14 ) -- ( 5.5,12 );
	\draw [style= thin] ( 2.5,17 ) -- ( 5.5,14 );
	\draw [style= thin] ( 2.5,17 ) -- ( 5.5,17 );
	\draw [style= thin] ( 2.5,17 ) -- ( 5.5,9 );
	\draw [style= thin] ( 2.5,17 ) ..controls(4.5, 18).. ( 7.5,17 );
	\draw [style= dashed] ( 2.5,17 ) to [out=-160,in=140] (1,8.5) to [out=-30,in=-160] ( 10.5,9);
	\draw [style= thin] ( 5.5,9 ) -- ( 5.5,12 );
	\draw [style= dashed] ( 5.5,9 ) -- ( 7.5,17 );
	\draw [style= dashed] ( 5.5,9 ) ..controls(6, 12).. ( 5.5,14 );
	\draw [style= thin] ( 5.5,9 ) ..controls(7.5, 8.3).. ( 10.5,9);
	\draw [style= thin] ( 5.5,12 ) -- ( 7.5,14 );
	\draw [style= dashed] ( 5.5,12 ) ..controls(7.5, 13).. ( 10.5,12 );
	\draw [style= dashed] ( 5.5,12 ) ..controls(6, 14.5).. ( 5.5,17 );
	\draw [style= thin] ( 5.5,14 ) -- ( 5.5,17 );
	\draw [style= dashed] ( 5.5,14 ) ..controls(7.5, 15).. ( 10.5,14 );
	\draw [style= thin] ( 5.5,14 ) -- ( 7.5,12 );
	\draw [style= thin] ( 5.5,17 ) ..controls(7.5, 18).. ( 10.5,17 );
	\draw [style= dashed] ( 5.5,17 ) -- ( 7.5,9 );
	\draw [style= dashed] ( 7.5,9 ) -- ( 10.5,9);
	\draw [style= thin] ( 7.5,9 ) -- ( 10.5,12 );
	\draw [style= thin] ( 7.5,9 ) -- ( 7.5,12 );
	\draw [style= thin] ( 7.5,9 ) -- ( 10.5,17 );
	\draw [style= dashed] ( 7.5,9 ) ..controls(7, 12.5).. ( 7.5,14 );
	\draw [style= thin] ( 7.5,12 ) -- ( 10.5,9);
	\draw [style= thin] ( 7.5,12 ) -- ( 10.5,14 );
	\draw [style= dashed] ( 7.5,12 ) -- ( 10.5,12 );
	\draw [style= dashed] ( 7.5,12 ) ..controls(6.5, 15).. ( 7.5,17 );
	\draw [style= dashed] ( 7.5,14 ) -- ( 7.5,17 );
	\draw [style= thin] ( 7.5,14 ) -- ( 10.5,14 );
	\draw [style= thin] ( 7.5,14 ) -- ( 10.5,17 );
	\draw [style= thin] ( 7.5,14 ) -- ( 10.5,12 );
	\draw [style= thin] ( 7.5,17 ) -- ( 10.5,17 );
	\draw [style= thin] ( 7.5,17 ) -- ( 10.5,14 );
	\draw [style= thin] ( 7.5,17 ) -- ( 10.5,9);
	\draw [style= thin] ( 10.5,9) -- ( 10.5,12 );
	\draw [style= dashed] ( 10.5,9) ..controls(11, 12).. ( 10.5,14 );
	\draw [style= dashed] ( 10.5,12 ) ..controls(11, 14.5).. ( 10.5,17 );
	\draw [style= dashed] ( 10.5,14 ) -- ( 10.5,17 );
	\end{tikzpicture}	
	
	${ (a). ~ AQ_4 = H \cup (AQ_4 - E(H))}$ \\

\end{center}

\begin{center}	
	\begin{tikzpicture}[scale=0.8]
	

	\draw [fill=black] (2,0) circle  (.08)  node [left]  at (2,0) {$1100$};
	\draw [fill=black] (2,1) circle  (.08)  node [right]  at (2,0.8) {$0100$};
	\draw [fill=black] (2,2) circle  (.08)  node [right]  at (2,1.8) {$0000$};
	\draw [fill=black] (2,3) circle  (.08)  node [left]  at (2,3) {$1111$};
	\draw [fill=black] (2,4) circle  (.08)  node [left]  at (2,4) {$1011$};
	\draw [fill=black] (2,5) circle  (.08)  node [left]  at (2,5) {$0011$};
	\draw [fill=black] (2,6) circle  (.08)  node [right]  at (2.05,5.8) {$0111$};
	\draw [fill=black] (2,7) circle  (.08)  node [left]  at (2,7) {$1000$};
	
	\draw [fill=black] (5,0) circle  (.08)  node [right]  at (5,0) {$1101$};
	\draw [fill=black] (5,1) circle  (.08)  node [left]  at (5,0.8) {$0101$};
	\draw [fill=black] (5,2) circle  (.08)  node [left]  at (5,1.8) {$0001$};
	\draw [fill=black] (5,3) circle  (.08)  node [right]  at (5,3) {$1110$};
	\draw [fill=black] (5,4) circle  (.08)  node [right]  at (5,4) {$1010$};
	\draw [fill=black] (5,5) circle  (.08)  node [right]  at (5,5) {$0010$};
	\draw [fill=black] (5,6) circle  (.08)  node [left]  at (5,5.8) {$0110$};
	\draw [fill=black] (5,7) circle  (.08)  node [right]  at (5,7) {$1001$};

	\draw [style= dashed] (2,0)--(2,1);
	\draw [style= dashed] (5,0)--(5,1);
	\draw [style= dashed] (2,1)--(2,2);
	\draw [style= dashed] (5,1)--(5,2);
	\draw [style= dashed] (2,2)--(2,7);
	\draw [style= dashed] (5,2)--(5,7);
	\draw [style= dashed] (2,0)--(5,0);
	\draw [style= dashed] (2,1)--(5,1);
	\draw [style= dashed] (2,2)--(5,2);
	\draw [style= dashed] (2,3)--(5,3);
	
	\draw [style= dashed] (2,5)--(5,5);
	\draw [style= dashed] (2,6)--(5,6);

	\draw [style= dashed] (2,7)--(5,4);
	\draw [style= dashed] (2,4)--(5,7);

	\draw [style= dashed] (2,0)..controls(0.5, 4)..(2,7);
	\draw [style= dashed] (5,0)..controls(6.5, 4)..(5,7);
	
	
	\draw [style= thin, color = gray] (8,0)--(11,3);
	\draw [style= thin, color = gray] (8,3)--(11,0);
	\draw [style= thin, color = gray] (8,2)--(11,6);
	\draw [style= thin, color = gray] (8,1)--(11,5);
	\draw [style= thin, color = gray] (8,5)--(11,2);
	\draw [style= thin, color = gray] (8,6)--(11,1);
	
	\draw [fill=black] (8,0) circle  (.08)  node [left]  at (8,0) {$0000$};
	\draw [fill=black] (8,1) circle  (.08)  node [left]  at (7.9,1) {$1000$};
	\draw [fill=black] (8,2) circle  (.08)  node [left]  at (7.9,2) {$1011$};
	\draw [fill=black] (8,3) circle  (.08)  node [left]  at (8,3) {$0100$};
	\draw [fill=black] (8,4) circle  (.08)  node [left]  at (8,4) {$0110$};
	\draw [fill=black] (8,5) circle  (.08)  node [left]  at (8,5) {$1110$};
	\draw [fill=black] (8,6) circle  (.08)  node [left]  at (7.9,6) {$1101$};
	\draw [fill=black] (8,7) circle  (.08)  node [left]  at (8,7) {$0010$};
	
	\draw [fill=black] (11,0) circle  (.08)  node [right]  at (11,0) {$0111$};
	\draw [fill=black] (11,1) circle  (.08)  node [right]  at (11.1,1) {$1111$};
	\draw [fill=black] (11,2) circle  (.08)  node [right]  at (11.1,2) {$1100$};
	\draw [fill=black] (11,3) circle  (.08)  node [right]  at (11,3) {$0011$};
	\draw [fill=black] (11,4) circle  (.08)  node [right]  at (11,4) {$0001$};
	\draw [fill=black] (11,5) circle  (.08)  node [right]  at (11,5) {$1001$};
	\draw [fill=black] (11,6) circle  (.08)  node [right]  at (11.1,6) {$1010$};
	\draw [fill=black] (11,7) circle  (.08)  node [right]  at (11,7) {$0101$};
	
	\draw [style= thick] (8,0)--(8,1);      
	\draw [style= thick] (11,0)--(11,1);
	\draw [style= thick] (8,1)--(8,2);
	\draw [style= thick] (11,1)--(11,2);
	\draw [style= thick] (8,2)--(8,7);
	\draw [style= thick] (11,2)--(11,7);
	\draw [style= thick] (8,0)--(11,0);
	\draw [style= thick] (8,1)--(11,1);
	\draw [style= thick] (8,2)--(11,2);
	\draw [style= thick] (8,3)--(11,3);
	\draw [style= thick] (8,4)--(11,4);
	\draw [style= thick] (8,5)--(11,5);
	\draw [style= thick] (8,6)--(11,6);
	\draw [style= thick] (8,7)--(11,7);
	
	\draw [style= thick] (8,0)..controls(6.5, 4)..(8,7);
	\draw [style= thick] (11,0)..controls(12.5, 4)..(11,7);
	
	\draw [style= thick] (8,7)--(11,4);
	\draw [style= thick] (8,4)--(11,7);

	\end{tikzpicture}
	
	\hspace{0.5in} {(b). $H$} \hspace{1in}  {(c). $AQ_4 - E(H)$}\\
	
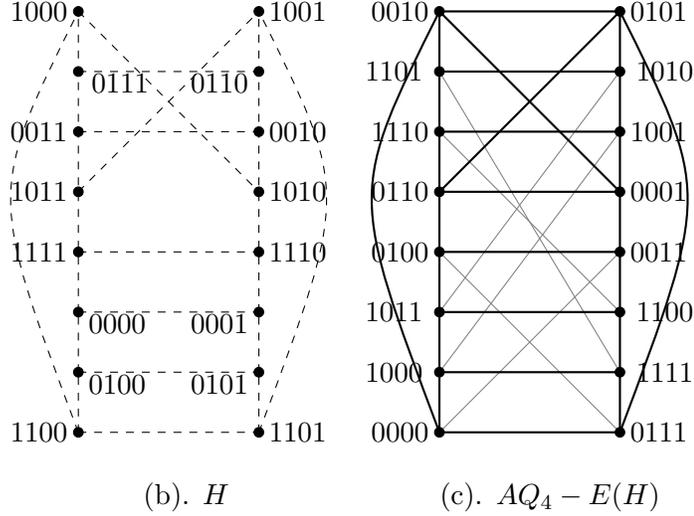
	\hspace{0.1in} {Figure 5: Decomposition of $AQ_4$ for $n_1 = 3$}\\
	
\end{center}
  We now prove (ii).  Obviously, $AQ_4 - E(H)$ is a 4-regular spanning subgraph of $AQ_4.$  We prove that $AQ_4 - E(H)$ is 4-connected. Let $W_1$ be a subgraph of $AQ_4$ induced by the eight vertices $0010,\,\, 1101,\,\, 1110,\,\, 0110,\,\, 0001,\,\, 1001,\,\, 1010,\,\, 0101$ and let $W_2$ be a subgraph of $AQ_4$ induced by the remaining vertices. Then $W_1$ is 3-regular and 3-connected and is isomorphic to $W_2.$  Observe that $AQ_4 - E(H)$ is union of  $W_1$ and $W_2$ and a perfect matching between them. Therefore, by Lemma \ref{l4}, $AQ_4 - E(H)$ is $4$-connected. 
  
  Let $L_2$ be the spanning subgraph of $AQ_4 - E(H)$ consisting of edges denoted by dark lines in  Figure 5(c). Then $L_2$ is a ladder-like  subgraph of $AQ_4 - E(H)$ containing a special 4-cycle $<0100,\, 1011,\, 1100,\, 0011,\, 0100>.$ Further, $L_2$ does not contain the edges $u_1 v_1$ and $u_4 v_4.$ This completes the proof.\end{proof}

We extend the above result for the general $n.$   
\begin{proposition}\label{p2} 		
	For $ n \geq 4,$ the augmented cube $AQ_n$ contains a spanning ladder-like subgraph $L_1$ satisfying the following properties. 
	\begin{enumerate} 
		\item $ H = L_1 -\{~u_1v_1~,~ u_4v_4~\}$ is 3-regular and 3-connected; and 
		\item $ AQ_n - E(H) $ is $(2n-4)$-regular and $(2n - 4)$-connected containing a spanning ladder-like subgraph $L_2$ with a special 4-cycle, which avoids the edges $u_1 v_1$ and $u_4 v_4.$
	\end{enumerate} 
\end{proposition}
\begin{proof}
	We proceed by induction on $n.$ By Lemma $ 4.1,$ the result holds for $n = 4.$ Suppose $n \geq 5.$ Suppose by induction $AQ_{n-1}^0$ contains  a spanning ladder-like subgraph $l_1$ such that $ H_0 = l_1 -\{~u_1v_1~,~ u_4v_4~\}$ is 3-regular and 3-connected and $ AQ_{n-1}^0 - E(H_0) $ is spanning, $(2n-6)$-regular, $(2n - 6)$-connected containing a spanning ladder-like subgraph $l_2$ with a special 4-cycle, which avoids the edges $u_1 v_1,\, u_4 v_4.$ Let $l'_i$ be the spanning ladder-like subgraph of $AQ_{n-1}^1$ corresponding to $l_i$  for $ i = 1, 2.$ Then $l_1' -\{~u'_1v'_1~,~ u'_4v'_4~\}$ is 3-regular and 3-connected. Further,  $(AQ_{n-1}^1 - E(l'_1)) \cup \{u'_1 v'_1, u'_4 v'_4\}$ is $(2n-6)$-regular, $(2n - 6)$-connected containing $l'_2$ which has a special 4-cycle. Moreover, $ l_2'$ contains none of the edges  $u'_1 v'_1$ and $ u'_4 v'_4.$  Note that $l_1'$ contains the two edges $u_1' v_4'$ and $u_4' v_1'.$ Let  $H_1 = l_1' - \{u_1' v_4',\, u_4' v_1'\}.$  Then $H_1$ is isomorphic to $C \times K_2$ for some cycle $C$ of length $2^{n-2}.$  Hence $H_1$ is a spanning 3-regular and 3-connected subgraph of $AQ_{n-1}^1.$ Therefore $AQ_{n-1}^1 - E(H_1)$ is a spanning, $(2n-6)$-regular and $(2n-6)$-connected subgraph of $AQ_{n-1}^1$ containing $l'_2.$  Define a subgraph $H$ as follows:	 
	 $$H = (H_0 - \{ u_t u_{t+1}, v_t v_{t+1}\}) \cup (H_1 - \{ u'_t u'_{t+1}, v'_t v'_{t+1}\}) \cup \{ u_t u'_t,\; v_t v'_t,\; u_{t+1} u'_{t+1},\; v_{t+1} v'_{t+1}\} ~~~{\textrm{(see~ Fig. 6.)}}$$

	 \begin{center} 
	 	\begin{tikzpicture}[scale=0.7]
	 	\draw [fill=black] (2,0) circle  (.08) node [left]  at (2,0) {$u_{m}$};
	 	\draw [fill=black] (2,1) circle  (.08) node [left]  at (1.9,1) {$u_{t+1}$}; 
	 	\draw [fill=black] (2,2) circle  (.08) node [left]  at (2.1,2) {$u_{t}$}; 
	 	\draw [fill=black] (2,3) circle  (.08) node [left]  at (2,3) {$u_{5}$}; 
	 	\draw [fill=black] (2,4) circle  (.08) node [left]  at (1.9,4) {$u_{4}$}; 
	 	\draw [fill=black] (2,5) circle  (.08) node [left]  at (2,5) {$u_3$}; 
	 	\draw [fill=black] (2,6) circle  (.08) node [left]  at (2,6) {$u_2$}; 
	 	\draw [fill=black] (2,7) circle  (.08) node [left]  at (2,7) {$u_1$}; 

	 	\draw [fill=black] (5,0) circle  (.08)  node [right]  at (5,0) {$v_{m}$};
	 	\draw [fill=black] (5,1) circle  (.08)  node [right]  at (5.2,1.1) {$v_{t+1}$};
	 	\draw [fill=black] (5,2) circle  (.08)  node [right]  at (4.9,1.8) {$v_{t}$};
	 	\draw [fill=black] (5,3) circle  (.08)  node [right]  at (4.95,2.8) {$v_5$};
	 	\draw [fill=black] (5,4) circle  (.08)  node [right]  at (5,4) {$v_{4}$};
	 	\draw [fill=black] (5,5) circle  (.08)  node [right]  at (5,5) {$v_3$};
	 	\draw [fill=black] (5,6) circle  (.08)  node [right]  at (5,6) {$v_2$};
	 	\draw [fill=black] (5,7) circle  (.08)  node [right]  at (5,7) {$v_1$};

	 	\draw [style= dotted] (2,0)--(2,1);
	 	\draw [style= dotted] (5,0)--(5,1);
	 
	 	\draw [style= dotted] (2,2)--(2,3);
	 	\draw [style= dotted] (5,2)--(5,3);
	 	\draw [style= thin ] (2,3)--(2,7);
	 	\draw [style= thin] (5,3)--(5,7);
	 
	 	\draw [style= thin] (2,0)--(5,0);
	 	\draw [style= thin] (2,1)--(5,1);
	 	\draw [style= thin] (2,2)--(5,2);
	 	\draw [style= thin] (2,3)--(5,3);
	
	 	\draw [style= thin] (2,5)--(5,5);
	 	\draw [style= thin] (2,6)--(5,6);

	 	\draw [style= thin] (2,0)..controls(1, 4)..(2,7);
	 	\draw [style= thin] (5,0)..controls(6, 4)..(5,7);
	 	
	 	\draw [style= thin] (2,7)--(5,4);
	 	\draw [style= thin] (2,4)--(5,7);

	 	
	 	\draw [fill=black] (8,0) circle  (.08) node [left]  at (8,0) {$u'_{m}$};
	 	\draw [fill=black] (8,1) circle  (.08) node [left]  at (7.9,1) {$u'_{t+1}$}; 
	 	\draw [fill=black] (8,2) circle  (.08) node [left]  at (8.2,1.6) {$u'_{t}$}; 
	 	\draw [fill=black] (8,3) circle  (.08) node [left]  at (8.05,2.9) {$u'_{5}$}; 
	 	\draw [fill=black] (8,4) circle  (.08) node [left]  at (7.85,4) {$u'_{4}$}; 
	 	\draw [fill=black] (8,5) circle  (.08) node [left]  at (8,5) {$u'_3$}; 
	 	\draw [fill=black] (8,6) circle  (.08) node [left]  at (8,6) {$u'_2$}; 
	 	\draw [fill=black] (8,7) circle  (.08) node [left]  at (8,7) {$u'_1$}; 

	 	\draw [fill=black] (11,0) circle  (.08)  node [right]  at (11,0) {$v'_{m}$};
	 	\draw [fill=black] (11,1) circle  (.08)  node [right]  at (11.1,1) {$v'_{t+1}$};
	 	\draw [fill=black] (11,2) circle  (.08)  node [right]  at (10.95,1.9) {$v'_{t}$};
	 	\draw [fill=black] (11,3) circle  (.08)  node [right]  at (11,3) {$v'_5$};
	 	\draw [fill=black] (11,4) circle  (.08)  node [right]  at (11,4) {$v'_{4}$};
	 	\draw [fill=black] (11,5) circle  (.08)  node [right]  at (11,5) {$v'_3$};
	 	\draw [fill=black] (11,6) circle  (.08)  node [right]  at (11,6) {$v'_2$};
	 	\draw [fill=black] (11,7) circle  (.08)  node [right]  at (11,7) {$v'_1$};

	 	\draw [style= dotted] (8,0)--(8,1);
	 	\draw [style= dotted] (11,0)--(11,1);
	 	\draw [style= dotted] (8,2)--(8,3);
	 	\draw [style= dotted] (11,2)--(11,3);
	 	\draw [style= thin ] (8,3)--(8,7);
	 	\draw [style= thin] (11,3)--(11,7);
	 	\draw [style= thin] (8,0)--(11,0);
	 	\draw [style= thin] (8,1)--(11,1);
	 	\draw [style= thin] (8,2)--(11,2);
	 	\draw [style= thin] (8,3)--(11,3);
	 	\draw [style= thin] (8,4)--(11,4);
	 	\draw [style= thin] (8,5)--(11,5);
	 	\draw [style= thin] (8,6)--(11,6);
	 	\draw [style= thin] (8,7)--(11,7);
	 	
	 	\draw [style= thin] (8,0)..controls(7, 4)..(8,7);
	 	\draw [style= thin] (11,0)..controls(12, 4)..(11,7);
	 	
	 	
	 	\draw [style= thin] (2,1)..controls(5.5, 0.1)..(8,1);
	 	\draw [style= thin] (2,2)..controls(5, 2.8)..(8,2);
	 	\draw [style= thin] (5,1)..controls(8.5, 0.1)..(11,1);
	 	\draw [style= thin] (5,2)..controls(8, 2.8)..(11,2);
	 	
	 	\end{tikzpicture}
	 	
	 	
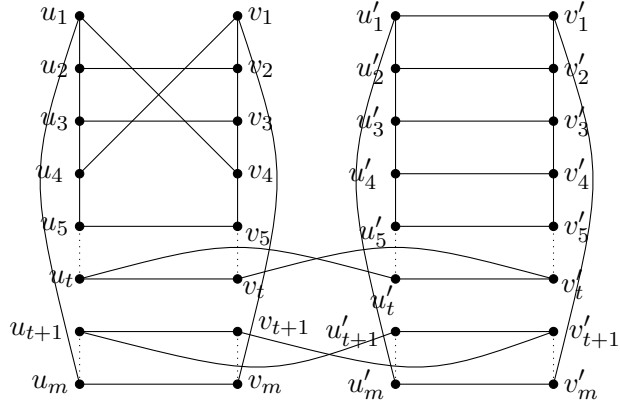
	 	\hspace{0.3in} {Figure 6: Construction of $H$ in $AQ_n$ for $n_1 = 3$}
	 \end{center}

	 Clearly, $H$ is a spanning ladder-like subgraph of $AQ_n$ without the edges $u_1 v_1$ and $u_4 v_4.$ By Lemma 2.3, $H$ is 3-regular and 3-connected. 
	 
	 Hence $AQ_n - E(H)$ is a spanning $(2n - 4)$-regular subgraph of $AQ_n.$   Let $K = AQ_n - E(H).$ We prove that $K$ contains a spanning ladder-like subgraph $L_2$ with a special 4-cycle. By induction hypothesis and construction of $K,$ it is easy to see that $K$ contain  a spanning ladder-like subgraph $l_2$ of $AQ^0_{n-1}$ and $l'_2$ that of $AQ^1_{n-1}.$ By Lemma \ref{l3}(ii), we get a spanning ladder-like subgraph $L_2$ with a special 4-cycle in $K$ using $l_2$ and $l'_2.$ 
	 
	 We now prove that $K$ is $(2n - 4)$-connected.  Note that 
	 $$K = K_0 \cup K_1 \cup E_n^c \cup (E^h_n - \{ u_t u'_t,\; v_t v'_t,\; u_{t+1} u'_{t+1},\; v_{t+1} v'_{t+1}\}),$$ 
	 where $K_0 = (AQ_{n-1}^0 - E(H_0)) \cup \{ u_t u_{t+1}, v_t v_{t+1}\}$ and
	 	 $K_1 = (AQ_{n-1}^1 - E(H_1)) \cup \{ u'_t u'_{t+1}, v'_t v'_{t+1}\}.$

	 Observe that there are $2^n - 4$ edges of $K$ between $K_0$ and $K_1.$  Since $ AQ_{n-1}^i - E(H_i)$ is $(2n - 6)$-connected for $ i = 0, 1,$ both $K_0$ and $K_1$ are $(2n - 6)$-connected.   Let $S$ be a subset of $V(K)$ with $|S| = 2n - 5.$ It suffices to prove that $K - S$ is connected. Suppose $ S \subset V(K_0).$  Observe that every component of $(K_0 - S)$ is joined by an edge of $E_n^h$ or $E_n^c$ to the connected graph $K_1.$ Hence $K - S$ is connected. Similarly, $K_S$ is connected if $ S \subset V(K_1). $  Suppose $ S$ intersects both $V(K_0) $ and $V(K_1). $ Then $S = S_0 \cup S_1,$ where $ S_0  = S \cap V(K_0),  $ $ S_1 = S \cap V(K_1).$   We may assume that $ |S_0| \geq |S_1|.$   Suppose $|S_0| < 2n - 6.$  Then $|S_1| < 2n - 6$ and therefore both $(K_0 - S_0)$ and $(K_1 - S_1)$ are connected and are joined to each other by an edge belonging to $E_n^c.$ Hence $(K - S)$ is connected.  Suppose $|S_0| = 2n - 6.$  Then $|S_1| = 1.$  Hence  $(K_1 - S_1)$ is connected. Let $D$ be a component of $ K_0 - S_0.$ Let $ v$ be a vertex of $D.$ If $ v \notin \{u_t, u_{t+1}, v_t ,v_{t+1} \},$ then $ v$ has at least two neighbours in $K_1.$  Suppose $ v \in \{u_t, u_{t+1}, v_t ,v_{t+1} \}.$ Then the degree of $v$ in $K_0$ is at least $2n -5$ and therefore it has at least one neighbour in $ K_0 - S_0,$ say $w.$ Then $ u$ and $w$ together have at least two neighbours in $K_2.$ Thus, in any case, the component $D$ has at least one neighbour in $K_1 - S_1.$  This implies that $ K- S$ has only one component and so it is connected.  
\end{proof}

 The case $n_1 = 3$ of Theorem \ref{2} follows from  Proposition \ref{p2} as the graph $H$ and its complement are 4-pancyclic by Lemma \ref{l2}.
%

\section{\textbf{Case $n_1 = 4$}}
In this section, we prove Theorem \ref{2} for the case $ n_1 = 4.$ To prove this we require the following lemma.

\begin{lem} [\cite{ma}] \label{l7}
	Any two vertices in $AQ_n$ have at most four common neighbours for $n \geq 3.$
\end{lem}
\begin{proposition}
For $n \geq 5,$ there exists a spanning, 4-regular, 4-connected subgraph $H$ of $AQ_{n}$ such that $AQ_n - E(H)$ is spanning, $(2n - 5)$-regular, $(2n - 5)$-connected. Moreover, $H$ contains a spanning ladder-like subgraph and $AQ_n - E(H)$ contains a spanning ladder-like subgraph with a special 4-cycle.
\end{proposition}
\begin{proof}
Write $AQ_n$ as $AQ_n = AQ_{n-1}^0 \cup AQ_{n-1}^1 \cup E_n^h \cup E_n^c.$ By Proposition 4.1, $AQ^0_{n-1}$ contains a ladder-like spanning subgraph $l_1$ such that if $H_0 = l_1 - \{ u_1 v_1, u_4 v_4 \},$ then $AQ_n - E(H_0)$ is spanning, $(2n - 6)$-regular, $(2n - 6)$-connected and contains a spanning ladder-like subgraph $l_2$ with a special 4-cycle, which avoids the edges $u_1 v_1$ and $u_4 v_4.$ Let $l'_1$ be the spanning ladder-like subgraph in $AQ^1_{n-1}$ corresponding to $l_1.$ Let $H_1 = l'_1 - \{ u'_1 v'_1, u'_4 v'_4 \}.$ Then $AQ^1_{n-1} - E(H_1)$ is a spanning, $(2n - 6)$-regular and $(2n - 6)$-connected subgraph of $AQ^1_{n-1}$. Let $l'_2$ be the spanning ladder-like subgraph of $AQ^1_{n-1} - E(H_1)$ corresponding to $l_2.$ Then $l'_2$ does not contain $u'_1 v'_1$ and $u'_4 v'_4.$

Let $H$ be a spanning subgraph of $AQ_n$ constructed from $l_1$ and $l'_1$ as follows.
Let $F_1 =\{u_2 u_2', u_3 u_3'\} \cup \{u_iu_i' \colon 5 \leq i \leq 2^{n-1} \}$ and let $F_2 =\{v_2 v_2', v_3 v_3'\} \cup \{v_iv_i' \colon 5 \leq i \leq 2^{n-1} \}.$ Define
$$H = l_1 \cup l'_1 \cup F_1 \cup F_2~~ \textrm{ (see Figure 7).}$$ 

Note that the edges $u_1 u'_1, u_4 u'_4, v_1 v'_1$ and $ v_4 v'_4$ do not belong to $H.$ 

\begin{center}
	\begin{tikzpicture}[scale=0.9]

\draw [fill=black] (2,3) circle  (.08) node [left]  at (2,3) {$u_{2^{n-2}}$}; 
\draw [fill=black] (2,4) circle  (.08) node [right]  at (1.9,3.8) {$u_6$}; 
\draw [fill=black] (2,5) circle  (.08) node [right]  at (2,4.8) {$u_5$}; 
\draw [fill=black] (2,6) circle  (.08) node [right]  at (2,5.8) {$u_4$}; 
\draw [fill=black] (2,7) circle  (.08) node [right]  at (2,6.8) {$u_3$}; 
\draw [fill=black] (2,8) circle  (.08) node [right]  at (2,7.8) {$u_2$}; 
\draw [fill=black] (2,9) circle  (.08) node [left]  at (2,9) {$u_1$};

\draw [fill=black] (5,3) circle  (.08)  node [right]  at (5.1,2.9) {$v_{2^{n-2}}$};
\draw [fill=black] (5,4) circle  (.08)  node [left]  at (5,3.8) {$v_6$};
\draw [fill=black] (5,5) circle  (.08)  node [left]  at (5,4.8) {$v_5$};
\draw [fill=black] (5,6) circle  (.08)  node [left]  at (5,5.8) {$v_4$};
\draw [fill=black] (5,7) circle  (.08)  node [left]  at (5,6.8) {$v_3$};
\draw [fill=black] (5,8) circle  (.08)  node [left]  at (5,7.8) {$v_2$};
\draw [fill=black] (5,9) circle  (.08)  node [right]  at (5,9) {$v_1$};

\draw [style= dotted] (2,3)--(2,4);
\draw [style= dotted] (5,3)--(5,4);
\draw [style= thin ] (2,4)--(2,5);
\draw [style= thin ] (5,4)--(5,5);
\draw [style= thin ] (2,5)--(2,9);
\draw [style= thin] (5,5)--(5,9);
\draw [style= thin] (2,3)--(5,3);
\draw [style= thin] (2,4)--(5,4);
\draw [style= thin] (2,5)--(5,5);
\draw [style= thin] (2,6)--(5,6);
\draw [style= thin] (2,7)--(5,7);
\draw [style= thin] (2,8)--(5,8);
\draw [style= thin] (2,9)--(5,9);

\draw [style= thin] (2,3)..controls(1, 6)..(2,9);
\draw [style= thin] (5,3)..controls(6, 6)..(5,9);

\draw [style= thin] (2,9)--(5,6);
\draw [style= thin] (2,6)--(5,9);


\draw [fill=black] (8,3) circle  (.08) node [left]  at (8.05,2.9) {$u'_{2^{n-2}}$}; 
\draw [fill=black] (8,4) circle  (.08) node [right]  at (7.85,3.8) {$u'_6$}; 
\draw [fill=black] (8,5) circle  (.08) node [right]  at (8,4.8) {$u'_5$}; 
\draw [fill=black] (8,6) circle  (.08) node [right]  at (8,5.8) {$u'_4$}; 
\draw [fill=black] (8,7) circle  (.08) node [right]  at (8,6.8) {$u'_3$}; 
\draw [fill=black] (8,8) circle  (.08) node [right]  at (8,7.8) {$u'_2$}; 
\draw [fill=black] (8,9) circle  (.08) node [left]  at (8,9) {$u'_1$};

\draw [fill=black] (11,3) circle  (.08)  node [right]  at (11,3) {$v'_{2^{n-2}}$};
\draw [fill=black] (11,4) circle  (.08)  node [left]  at (11,3.8) {$v'_6$};
\draw [fill=black] (11,5) circle  (.08)  node [left]  at (11,4.8) {$v'_5$};
\draw [fill=black] (11,6) circle  (.08)  node [left]  at (11,5.8) {$v'_4$};
\draw [fill=black] (11,7) circle  (.08)  node [left]  at (11,6.8) {$v'_3$};
\draw [fill=black] (11,8) circle  (.08)  node [left]  at (11,7.8) {$v'_2$};
\draw [fill=black] (11,9) circle  (.08)  node [right]  at (11,9) {$v'_1$};

\draw [style= dotted ] (8,3)--(8,4);
\draw [style= dotted] (11,3)--(11,4);
\draw [style= thin ] (8,4)--(8,5);
\draw [style= thin] (11,4)--(11,5);
\draw [style= thin ] (8,5)--(8,9);
\draw [style= thin] (11,5)--(11,9);
\draw [style= thin] (8,3)--(11,3);
\draw [style= thin] (8,4)--(11,4);
\draw [style= thin] (8,5)--(11,5);
\draw [style= thin] (8,6)--(11,6);
\draw [style= thin] (8,7)--(11,7);
\draw [style= thin] (8,8)--(11,8);
\draw [style= thin] (8,9)--(11,9);

\draw [style= thin] (8,3)..controls(7, 6)..(8,9);
\draw [style= thin] (11,3)..controls(12, 6)..(11,9);

\draw [style= thin] (8,9)--(11,6);
\draw [style= thin] (8,6)--(11,9);

\draw [style= thin] (2,3)..controls(5, 3.5)..(8,3);
\draw [style= thin] (2,4)..controls(5, 4.5)..(8,4);
\draw [style= thin] (2,5)..controls(5, 5.5)..(8,5);
\draw [style= thin] (2,7)..controls(5, 7.5)..(8,7);
\draw [style= thin] (2,8)..controls(5, 8.5)..(8,8);
\draw [style= thin] (5,3)..controls(8, 3.5)..(11,3);
\draw [style= thin] (5,4)..controls(8, 4.5)..(11,4);
\draw [style= thin] (5,5)..controls(8, 5.5)..(11,5);
\draw [style= thin] (5,7)..controls(8, 7.5)..(11,7);
\draw [style= thin] (5,8)..controls(8, 8.5)..(11,8);

\end{tikzpicture}

Figure 7: 4-regular, 4-connected subgraph $H$ of $AQ_n$\\
\end{center}

From Figure 6, it is clear that $H$ is a spanning, 4-regular subgraph of $AQ_n.$ By Lemma \ref{l3}(i), we get a spanning ladder-like subgraph $L$ in $H$ using $l_1$ and $l_1'.$
 By Corollary 2.1, $l_1$ and $l_1'$ are 3-connected. Now we prove that $H$ is $4$-connected. Let $S \subset V(H)$ with $|S| = 3.$ Let $S = S_1 \cup S_2,$ where $S_1 \subset V(l_1)$ and $S_2 \subset V(l'_1).$ We may assume that $|S_1| \geq |S_2|.$ Suppose $|S_1| < 3$ and $|S_2| < 3.$ As $l_1$ and $l'_1$ are 3-connected, both $l_1 - S_1$ and $l'_1 - S_2$ are connected. Since there are $2^n - 4$ edges from $E_n^h$ in between $l_1$ and $l'_1,\, H - S$ is connected. Suppose $|S_1| = 3.$ Then $S_2 = \emptyset.$ Every vertex of $l_1 \setminus \{ u_1, u_4, v_1, v_4\}$ is connected to $l'_1$ by an edge from $E_n^h$. Observe that each of $\{ u_1, u_4, v_1, v_4\}$ have four neighbours in $l_1$ and hence each of $\{ u_1, u_4, v_1, v_4\}$ has at least one neighbour in $l_1 - S_1.$ Through this neighbour they are connected with $l'_1.$ Therefore $H - S$ is connected. Thus, $H$ is $4$-connected.

Observe that $AQ_n - E(H)$ is spanning and $(2n - 5)$-regular. Also, by Lemma \ref{l3}(ii), we get a spanning ladder-like subgraph with a special 4-cycle in $AQ_n - E(H)$ from $l_2$ and $l'_2.$

It remains to prove that $AQ_n - E(H)$ is $(2n - 5)$-connected. Let $G = AQ_n - E(H),\, G_1 = AQ^0_{n-1} - l_1$ and $G_2 = AQ^1_{n-1} - l'_1.$ Then, we have $ G = G_1 \cup G_2 \cup \{u_1 u_1',\, v_1 v_1',\, u_4 u_4',\, v_4 v_4' \} \cup E^c_n.$ Since $(AQ^0_{n-1} -E( l_1)) \cup \{u_1 v_1, u_4 v_4\}$ and $(AQ^1_{n-1} - E(l'_1)) \cup \{u'_1 v'_1, u'_4 v'_4\}$ are $(2n - 6)$-connected, the graphs $G_1$ and $G_2$ are $(2n - 8)$-connected. Let $S \subset V(G)$ with $|S| = 2n - 6.$  It is sufficient to prove that $ G - S$ is connected.

 Let $ U = \{u_1, v_1, u_4, v_4\}$ and let $x \in V(G_1).$ If $x \in U,$ then the degree of $x$ in $G_1$ is $2n - 7$ and it has two neighbours in $G_2.$ If $x \notin U,$ then the degree of $x$ in $G_1$ is $2n - 6$ and it has only one neighbour in $G_2.$ If $ S \subset V(G_1)$ or $ S \subset V(G_2),$ then $ G - S$ is obviously connected. 

Suppose $S = S_1 \cup S_2,$ where $S_1 \subset V(G_1)$ and $S_2 \subset V(G_2)$ with $2n-7\geq |S_1| \geq |S_2| \geq 1.$  Suppose  $|S_1|< 2n-8.$  Then $|S_1|< 2n-8.$  Therefore $G_1 - S_1$ and $ G_2 - S_2$ are connected and joined to each other by an edge of $G.$ Thus $ G - S$ is connected.  

Suppose $ |S_1| = 2n-8$ or $ 2n - 7.$ Then $|S_2| \leq 2.$  Then $ G_2 - S_2$ is connected since $G_2$ contains a ladder-like subgraph $l_2'$ which is 3-connected.  Let $D$ be a component of  $ G_1 - S_1.$ 
We prove that $D$ has a neighbour in the connected graph $G_2 - S_2.$  The minimum degree of $D$ is at least one. If $D$ has more than two vertices, then $D$ has at least three neighbours in $G_2$ and so it has a neighbour in $G_2 - S_2.$ Suppose $D$ has only two vertices, say $u$ and $v.$ Clearly, $D$ is an edge $uv.$ Then $D$ has at least two neihgbours in $G_2$ and so has a neighbour in $G_2 - S_2$ if $|S_2| = 1.$ Therefore we may assume that $|S_1| = 2n - 8$ and $|S_2| = 2.$ This implies that both $u$ and $v$ belong to $U$ and further, every vertex in $S_1$ is a common neighbour of both.  If $n = 5,$ then from Figure 5(b), we have existence of $H$ in $AQ_5$ such that any two adjacent vertices in $U,$ together have four neighbours in $G_2$ and so $D$ has a neighbour in $G_2 - S_2$ in $AQ_5.$ Suppose $n \geq 6.$ Then $u$ and $v$ have at least $|S_1| = 2n - 8 \geq 4$ common neighbours in $G_1.$ By Lemma \ref{l7}, $u$ and $v$ can not have more than four common neighbours in graph $AQ_n.$ Hence the neighbours of $u$ and $v$ in $G_2$ are all distinct and therefore, $D$ has at least one neighbour in $G_2 - S_2.$ Thus $G - S$ is connected.

Therefore $G$ is $(2n - 5)$-connected. This completes the proof.
\end{proof}	

 The case $n_1 = 4$ of Theorem \ref{2} follows from the above proposition as the graph $H$ and its complement are 4-pancyclic by Lemma \ref{l2}.

\section {General Case}

\begin{proposition}
Let $n \geq 4$ and $2n - 1 = n_1 + n_2$ with $n_1, n_2 \geq 2.$ Then the augmented cube $AQ_n$ can be decomposed into two spanning subgraphs $H$ and $K$ such that $H$ is $n_1$-regular and $n_1$-connected and $K$ is $n_2$-regular and $n_2$-connected. Further, $H$ contains a spanning ladder-like subgraph and $K$ contains a spanning ladder-like subgraph with a special 4-cycle if $n_1, n_2 \geq 4.$
\end{proposition}

\begin{proof}	
We prove the result by induction on $n.$ We may assume that $n_1 \leq n_2.$ Suppose $n = 4.$ Then $n_1 = 2$ or 3 the result holds. By Propositions 3.1, 4.1 and 4.2 the result is true for $n_1 = 2, 3, 4.$  Hence the result holds for $n = 4$ and $n = 5.$ Suppose $n \geq 6$ and $n_1 \geq 5.$

By induction hypothesis, $AQ_{n-1}^0$ can be decomposed into two spanning subgraphs $H_0$ and $K_0$ such that $H_0$ is $(n_1 - 1)$-regular, $(n_1 - 1)$-connected and $K_0$ is $(n_2 - 1)$-regular, $(n_2 - 1)$-connected. Further, $H$ contains a spanning ladder-like subgraph, say $l_1$  and $K$ contains a spanning ladder-like subgraph with a special 4-cycle, say $l_2$. Let $H_1$ and $K_1$ be the corresponding spanning subgraphs of $AQ_{n-1}^1.$ Let $l'_1$ and $l'_2$ be the corresponding spanning ladder-like subgraphs of $H_1$ and $K_1,$ respectively. 

We can write $AQ_n = AQ_{n-1}^0 \cup AQ_{n-1}^1 \cup E_n^h \cup E_n^c.$

Define $H = H_0 \cup H_1 \cup E_n^h$ and $K = K_0 \cup K_1 \cup E_n^c.$

Clearly, $H$ is $n_1$-regular and $K$ is $n_2$-regular and further, both are spanning subgraphs of $AQ_n.$ By Lemma \ref{l4}, $H$ is $n_1$-connected and $K$ is $n_2$-connected. Now, by Lemma \ref{l3}(i), we get a spanning ladder-like subgraph $L_1$ in $H$ from $l_1,\; l'_1$ and using four edges of $E_n^h.$ Similarly, by Lemma \ref{l3}(ii), we get a spanning ladder-like subgraph $L_2$ in $K$ with a special 4-cycle, from $l_2,\; l'_2$ and using four edges of $E_n^c.$
\end{proof}	

\textbf{ Proof of Main Theorem 1.2.}
\begin{proof} 
We may assume that $n_1 \leq n_2.$ The result holds for the case $n_1 = 2$ by Proposition 3.2 and Lemma 2.3. Suppose $n_1 = 3.$ Then, by Proposition 4.2, $AQ_n$ has spanning ladder-like subgraph $L$ such that $ H = L -\{~u_1v_1~,~ u_4v_4~\}$ is 3-regular and 3-connected and its complement $ AQ_n - E(H) $ is $(2n-4)$-regular and $(2n - 4)$-connected containing a spanning ladder-like subgraph with a special 4-cycle. Thus $H$ and $ AQ_n - E(H) $ are spanning subgraphs of $AQ_n$ and, by Lemma 2.3, they are 4-pancyclic. Suppose $n_1 \geq 4.$ Then $n_2 \geq 4.$ Now the result follows from Proposition 6.1 and Lemma 2.3. 
\end{proof}

\vspace{0.5cm}

\noindent
{\bf Concluding Remarks. } The main theorem of the paper guarantees the existence of a decomposition of $AQ_n$ into two spanning, regular, connected and pancyclic  subgraphs, whose degrees correspond to the parts of the given 2-partition of the degree $2n-1$ of $AQ_n.$
This result can be generalized to the decomposition of $AQ_n$ into $k$ subgraphs according to the given $k$-partition of $n.$ In particular, the problem of decomposing $AQ_n$ into Hamiltonian cycles and a perfect matching is still open.

 We also note that the main theorem of the paper provides a partial solution to the following question due to Mader [\cite{md}, pp.73].

\vspace{0.5cm}
\noindent
{\bf Question}(\cite{md}). {\it Given any $n$-connected graph and $k \in 	\{1, 2,..., n\}$ is there always a $k$-connected subgraph $H$ of $G$ so that $ G - E(H)$ is $( n - k)$-connected? }


\begin{thebibliography}{99}
		\bibitem{al} {B. Alspach, J.-C. Bermond and D. Sotteau}, Decomposition into cycles I: Hamilton decompositions, Proceedings of NATO Advanced Research Workshop on Cycles and Rays $(1990) 9-18.$
		
		\bibitem{ba} { D. W. Bass and I. H. Sudborough}, Hamiltonian decompositions and $(n/2)$-factorizations of hypercubes, J. Graphs Algorithms Appl. $7 no. 1 (2003) 79-98.$
		
		\bibitem{bk} {Y. M. Borse and S. A. Kandekar}, Decomposition of hypercubes into regular connected bipancyclic subgraphs, Discrete Math. Algorithms Appl. $7 no. 3 (2015) Article 1550033 10 pp.$
		
		\bibitem{bs} {Y. M. Borse  and S. R. Shaikh}, Decomposition of the product of cycles based on degree partition, Discuss.  Math. Graph Theory (to appear).
		
		\bibitem{bss} {Y. M. Borse, A. V. Sonawane and S. R. Shaikh}, Connected bipancyclic isomorphic $m$-factorizations of the Cartesian product of the graphs, Australas. J. Combin. $66(1) (2016) 120-129.$
		
		\bibitem{ch}{D. Cheng, R.-X. Hao, Y.-Q. Feng}, Conditional edge-fault pancyclicity of augmented cubes, Theoret. Comput. Sci. $510 (2013) 94-101.$
		
		\bibitem{cs} {S. A. Choudum, V. Sunitha}, Augmented cubes, Networks $40(2) (2002) 71-84$.
		
		\bibitem{fu} {J.-S. Fu}, Vertex-pancyclicity of augmented cubes with maximal faulty edges, Inform. Sci. $275 (2014) 257-266.$
		
		\bibitem{hs} {S.-Y. Hsieh, J.-Y. Shiu}, Cycle embedding of augmented cubes, Appl. Math. Comput. $191 (2007) 314-319.$
				
		\bibitem{md} { W. Mader},  Connectivity and edge-connectivity in finite graphs. in Surveys in Combinatorics, ed. B. Bollobas, London Math. Soc. Lecture Note Ser. $Vol. 38 (1979) pp.66-95.$
		
		\bibitem{ml} {M. Ma, G. Liu, J.-M. Xu}, Panconnectivity and edge-fault-tolerant pancyclicity of augmented cubes, Parallel Comput. $33 (2007) 36-42.$
		
		\bibitem{ma} {M. Ma, G. Liu and J.-M. Xu}, The super connectivity of augmented cubes, Inform. Process. Lett. $106(2) (2008) 59-63.$
		
		\bibitem{sx} {M. Ma, Y. Song, J.-M. Xu}, Fault Tolerance of Augmented Cubes, AKCE Int. J. Graphs Comb. $10 1(2013) 37-55.$
				
		\bibitem{an} { S. A. Mane},  Domination Parameters and Fault Tolerance of Hypercubes, Ph.D. Thesis, Pune University $2012.$
		
		\bibitem{ss} {S. A. Mane, S. A. Kandekar and B. N. Waphare}, Constructing spanning trees in augmented cubes, J. Parallel Distrib. Comput. $(2018) https://doi.org/10.1016/j.jpdc.2018.08.006.$
		
		\bibitem{sm} { S. A. Mane, B. N. Waphare}, Regular connected bipancyclic spanning subgraphs of hypercubes, Comput. Math. Appl. $62 (2011) 3551-3554.$
		
		\bibitem{mo} { M. Mollard and M. Ramras}, Edge decompositions of hypercubes by paths and by cycles, Graphs Combin. $31 no. 3 (2015) 729–741.$
		  
		
		\bibitem{sb} {A. V. Sonawane and Y. M. Borse}, Decomposing hypercubes into regular connected subgraphs, Discrete Math. Algorithms Appl. $7 no. 4 (2016), Article 1650065, 6 pp.$
		
		\bibitem{wm} {W.-W. Wang, M.-J. Ma, J.-M. Xu}, Fault-tolerant pancyclicity of augmented cubes, Inform. Process. Lett. $103 (2007) 52-56.$
				
		\bibitem{wa} { S. G. Wagner and M. Wild}, Decomposing the hypercube $Q_n$ into $n$ isomorphic edge-disjoint trees, Discrete Math. $312 no.10 (2012) 1819-1822.$
		
		\bibitem{ww} {H.-L. Wang, J.-W. Wang, J.-M. Xu}, Fault-tolerant panconnectivity of augmented cubes, Front. Math. China $2009 4(4) 697-719.$
		
		\bibitem{we} {D. B. West}, Introduction to Graph Theory, Pearson Education, Delhi $2001.$ 
		

\end{thebibliography}
\end{document}